\documentclass[a4paper,11pt,reqno]{amsart}
\usepackage{amsmath, amscd, amssymb, hyperref}
\usepackage{tabularx}
\usepackage{tabulary}

\usepackage{euscript}
\usepackage{epsfig}

\usepackage[alphabetic,initials]{amsrefs}

\begin{document}

\theoremstyle{plain} \numberwithin{equation}{section}

\newtheorem{thm}{Theorem}[section]
\newtheorem{prop}[thm]{Proposition}
\newtheorem{cor}[thm]{Corollary}
\newtheorem{conj}[thm]{Conjecture}
\newtheorem{lemma}[thm]{Lemma}
\newtheorem{lemdef}[thm]{Lemma--Definition}
\newtheorem{ques}[thm]{Question}
\newtheorem{claim}[thm]{Claim}

\theoremstyle{definition}

\newtheorem{defn}[thm]{Definition}
\newtheorem{ex}[thm]{Example}
\newtheorem{examples}[thm]{Examples}
\newtheorem{notn}[thm]{Notation} 
\newtheorem{note}[thm]{Note}
\newtheorem{question}[thm]{Question}
\newtheorem{obs}[thm]{Observation}
\newtheorem{rmk}[thm]{Remark}
\newtheorem{review}[thm]{}
\newtheorem{Empty}[thm]{}
\newcommand{\bi}{\begin{itemize}}  
\newcommand{\ei}{\end{itemize}}
\newcommand{\bp}{\begin{proof}}
\newcommand{\ep}{\end{proof}}

\def\Trick#1{\begin{Empty}\bf#1\end{Empty}}

\def\AA{\mathbb{A}}
\def\CC{\mathbb{C}}
\def\FF{\mathbb{F}}
\def\GG{\mathbb{G}}
\def\GGG{\mathbb{G}}
\def\PP{\mathbb{P}}
\def\cPP{\check{\mathbb{P}}}
\def\SS{\mathbb{S}}
\def\WW{\mathbb{W}}

\def\QQ{\mathbb{Q}}
\def\RR{\mathbb{R}}
\def\VVV{\mathbb{V}}
\def\ZZ{\mathbb{Z}}
\def\NN{\mathbb{N}}
\def\VV{\mathbb{V}}

\def\m{\mathfrak{m}}

\def\arr{\mathop{\longrightarrow}}

\def\ov{\overline}

\def\bG{\Bbb G}
\def\bP{\Bbb P}
\def\al{\alpha}
\def\be{\beta}
\def\de{\delta}
\def\eps{\epsilon}
\def\ga{\gamma}
\def\io{\iota}
\def\ka{\kappa}
\def\la{\lambda}
\def\na{\nabla}
\def\om{\omega}
\def\si{\sigma}
\def\th{\theta}
\def\ups{\upsilon}
\def\ve{\varepsilon}
\def\vp{\varpi}
\def\vt{\vartheta}
\def\ze{\zeta}

\def\De{\Delta}
\def\Ga{\Gamma}
\def\La{\Lambda}
\def\Om{\Omega}

\def\cA{\mathcal{A}}
\def\cB{\mathcal{B}}
\def\cC{\mathcal{C}}
\def\cD{\mathcal{D}}
\def\cE{\mathcal{E}}
\def\cF{\mathcal{F}}
\def\cG{\mathcal{G}}
\def\cH{\mathcal{H}}
\def\cI{\mathcal{I}}
\def\cJ{\mathcal{J}}
\def\cK{\mathcal{K}}
\def\cL{\mathcal{L}}
\def\cM{\mathcal{M}}
\def\cN{\mathcal{N}}
\def\cO{\mathcal{O}}
\def\cT{\mathcal{T}}
\def\cS{\mathcal{S}}
\def\cU{\mathcal{U}}
\def\cHom{\mathcal{H}\text{\it om}}

\def\ocM{\overline{\mathcal{M}}}

\def\tX{\tilde X}
\def\tY{\tilde Y}
\def\tN{\tilde N}

\def\bA{\overline{A}}
\def\bB{\overline{B}}
\def\bE{\overline{E}}
\def\bF{\overline{F}}
\def\bL{\overline{L}}
\def\bM{\overline{M}}
\def\bZ{\overline{Z}}
\def\bbf{\overline{f}}
\def\bg{\overline{g}}

\def\c{\text{c}}
\def\ch{\text{ch}}
\def\codim{\text{codim}}
\def\dim{\text{\rm dim}}
\def\Supp{\text{\rm Supp}}
\def\Eff{\text{\rm Eff}}
\def\Exc{\text{\rm Exc}}
\def\Ext{\text{\rm Ext}}
\def\Sym{\text{\rm Sym}}
\def\Aut{\text{\rm Aut}}
\def\Hom{\text{\rm Hom}}
\def\Coker{\text{\rm Coker}}
\def\Ker{\text{\rm Ker}}
\def\Tor{\text{\rm Tor}}
\def\Cone{\text{\rm Cone}}
\def\h{\text{h}}
\def\inf{\infty}
\def\p{\mathfrak{p}}
\def\HH{\text{H}}
\def\Id{\text{Id}}
\def\LM{\overline{\text{LM}}}
\def\M{\overline{\text{M}}}
\def\NE{\overline{\text{NE}}}
\def\PGL{\text{\rm PGL}}
\def\SL{\text{\rm SL}}
\def\Pic{\text{\rm Pic}}
\def\Proj{\text{\rm Proj}}
\def\R{\text{R}}
\def\Spec{\text{Spec}}
\def\td{\text{td}}
\def\Bl{\text{\rm Bl}}
\def\wt{\text{weight}}

\def\dra{\dashrightarrow}
\def\hra{\hookrightarrow}
\def\lra{\leftrightarrow}
\def\ra{\rightarrow}

\newcommand{\sslash}{\mathbin{/\mkern-6mu/}}


\title{Derived category of moduli of pointed curves - I}

\author{Ana-Maria Castravet and Jenia Tevelev}

\address{Universit\'e Paris-Saclay, UVSQ, CNRS, Laboratoire de Math\'ematiques de Versailles, 78000, Versailles, France}
\email{ana-maria.castravet@uvsq.fr }

\address{Department of Mathematics and Statistics, University of Massachusetts Amherst, 710 North Pleasant Street, Amherst, MA 01003 
and Laboratory of Algebraic Geometry and its applications, HSE, Moscow, Russia}
\email{tevelev@math.umass.edu}
\thanks{The first named author was supported by NSF grants DMS-1529735 and DMS-1701752. 
The second named author was supported by NSF grants DMS-1303415 and DMS-1701704.}

\begin{abstract}
This is the first paper in the sequence devoted to  derived category of moduli spaces of curves of genus $0$  with marked points.
We develop several approaches to describe it equivariantly with respect to the action of the symmetric group
permuting marked points.
We construct an equivariant  full exceptional collection on the Losev--Manin space
which categorifies derangements. 
\end{abstract}

\maketitle


\section{Introduction}

The special feature of moduli spaces of curves with marked points is the action of the symmetric group permuting marked points, and our goal is to exhibit this action
in the description of the derived category. One can think about the derived category as an enhanced cohomological invariant.
Although there are many papers in the literature computing
cohomology of $\M_{0,n}$ as a module over the symmetric group (e.g.~\cite{Getzler, BergstromMinabe}),
the equivariant Euler--Poincare polynomial is expressed as an alternating sum, which therefore has 
no obvious categorification. 
On the other hand, it is  often  easy to get some description 
of the derived category which however does not respect the group action.
For example, it is obvious that $D^b(\M_{0,n})$ has a full exceptional
collection. Indeed, $\M_{0,n}$ has a Kapranov model as an
iterated blow-up of $\bP^{n-3}$ in $n-1$ points followed by the blow-up of $n-1\choose 2$ proper transforms of lines connecting points, etc.
With a little work, Orlov's theorem on derived category of the blow-up (see \S\ref{zfasfhasfh}) gives a  full exceptional collection.
However, Kapranov's model is not unique: it depends on the choice of the $\psi$ class, i.e.,~the choice of a marking, and therefore
this collection is not preserved by $S_n$ (only by $S_{n-1}$).
The derived categories of $\M_{0,n}$ and related Hassett spaces and GIT quotients have been   studied in \cite{BFK} and \cite{MS}, 
although not from the equivariant perspective.

\begin{question}\label{Orlov question}
Is there a full exceptional  $S_n$-invariant collection on $\M_{0,n}$?
\end{question}

This question of D.~Orlov, communicated to us by A. Kuznetsov, will be investigated in detail in the second paper in the series.
Note that a striking and unexpected corollary of its existence is that the K-group  $K_0(\M_{0,n})$
is  a {\em permutation} representation of $S_n$. As a motivation, one can argue that since $\M_{0,n}$ is smooth over $\Spec\,\ZZ$,
maybe it is somehow ``defined over $\FF_1$'', and therefore the same should be true of its  K-theory as an $S_n$-module,
and so perhaps it should be a permutation representation.

In this paper we suggest two general strategies which may have other  applications 
and provide an answer for the Losev--Manin space \cite{LM}.

\medskip 

One approach, which justifies why consider the case of Losev-Manin spaces, 
is based on an equivariant version of Orlov's theorem on blow-ups (Section \S\ref{equiOrlovS}) and inspired by 
the work of Bergstrom and Minabe in \cite{BergstromMinabe}. 

Let $X$ be a smooth projective variety and let 
$Y_1,\ldots,Y_n\subseteq X$ be smooth transversal subvarieties of codimension $l$. For any subset 
$I\subseteq\{1,\ldots,n\}$, 
we denote by $Y_I$ the intersection $\cap_{i\in I}Y_i$. In particular, $Y_\emptyset=X$.
Let $q:\,\tX\to X$ be an iterated blow-up of (proper transforms of) $Y_1,\ldots,Y_n$.
In addition, let $G$ be a finite group acting on $X$ permuting $Y_1,\ldots,Y_n$.
Then it also acts on $\tilde X$ and the morphism $q$ is $G$-equivariant. Let $G_I\subseteq G$ be a normalizer of $Y_I$ for each subset 
$I\subseteq \{1,\ldots,n\}$ (in particular, $G_\emptyset=G$). 
We show in Lemma~\ref{BigDaddy} that if 
$D^b(Y_I)$ admits a full $G_I$-equivariant exceptional collection for every subset $I$
then $D^b(\tilde X)$ admits a full $G$-equivariant exceptional collection.

Next we generalize an inductive computation in \cite{BergstromMinabe} of the equivariant Euler--Poincare polynomial
 of $\M_{0,n}$. In the derived category setting, it gives the following theorem.
 Fix integers $l\ge1$ and $0\le k\le n$.
 For a weight 
$${\bold a}=\left(1, . . . ,1, {1\over l},\ldots, {1\over l}\right)$$ 
(with $k$ copies of $1$ and $n-k$ copies of $1\over l$), let 
$\M^n_{k,l}$ be the Hassett moduli space \cite{Ha} of $\bold a$-weighted stable rational curves.
For example, $\M^n_{0,1}\simeq\M_{0,n}$ and  
$$\M^n_{0,\lfloor {n-1\over 2}\rfloor}$$ is a symmetric GIT quotient $(\PP^1)^n\sslash\PGL_2$
if $n$ is odd and its Kirwan resolution if $n$ is even.

\begin{thm}\label{BMapproach}
If $\M^n_{k,r(n,k)}$ admits a full $(S_k\times S_{n-k})$-equivariant exceptional collection for every $n$ and every $0\le k\le n-3$ then
$\M_{0,n}$ admits a full $S_n$-equivariant exceptional collection for every $n$. Here
$$r(n, k) :=\begin{cases}
\left\lfloor {n-1\over 2}\right\rfloor & \mbox{if } k=0\\
n-2 & \mbox{if } k=1\\
n-k & \mbox{if } k\ge2.\\
\end{cases}
$$
\end{thm}

Concretely, we need the following spaces:
\begin{itemize}
\item The symmetric GIT quotient and its Kirwan resolution, which will  be studied in the sequel to this paper.
\item $\M^n_{1,n-2}\simeq\PP^{n-3}$ via Kapranov map. We can take any standard exceptional 
collection on $\PP^{n-3}$, for example $\cO,\ldots,\cO(n-3)$.
\item $\M^n_{2,n-2}$. This is the Losev-Manin space studied in this paper.
\item Spaces $\M^n_{k,n-k}$ for $k>2$. These spaces are still two complicated for the calculations of the derived category
and in the sequel to this paper
we will investigate their further equivariant reductions.
\end{itemize}


\

We now discuss another strategy, which is the one we will use in this paper for the case of the Losev-Manin spaces $\LM_n$.  
We start with an example.
\begin{ex}
Unlike $\M_{0,5}$, which has $5$ Kapranov models and therefore $5$ Orlov-style exceptional collections,
the $2$-dimensional Losev--Manin space, which we denote by $\LM_3$ in this paper (see below),
has only two non-trivial $\psi$-classes $\psi_0$ and $\psi_\infty$, realizing it as $\bP^2$ blown-up at three points $p_1$, $p_2$, $p_3$
in two ways, related by the Cremona involution. The corresponding exceptional collection invariant under
all automorphisms has three blocks and consists of line bundles
\begin{equation}
\{-\psi_0,-\psi_\infty\}, \{\pi_1^*\cO(-1), \pi_2^*\cO(-1), \pi_3^*\cO(-1)\}, \cO,
\label{afsvaf}\end{equation}
where $\pi_i:\, \LM_3\to\LM_2\simeq\bP^1$
is a forgetful map, which can be thought of as a linear projection $\bP^2\dra\bP^1$ from the point $p_i$.
\end{ex}

The last four line bundles in \eqref{afsvaf} are pull-backs under  forgetful maps but the first two have a trivial derived pushforward
by any forgetful map. To study  situations of this sort more systematically, we introduce
an inclusion-exclusion principle in triangulated categories (see Lemma~\ref{jmgfgk}) and its application in the following set-up.

\begin{defn}\label{zxfbdfdfhdfd}
Given a collection of morphisms of smooth projective varieties $\pi_i:\,X\to X_i$ for $i\in I$,
we call an object 
$E\in D^b(X)$ {\em cuspidal}
\footnote{The terminology (suggested to us by Alex Oblomkov) comes from cuspidal representations of representation theory.
When considering a single morphism, this is sometimes known as the \emph{null-category}.} if 
$${R\pi_i}_*E=0\quad \hbox{\rm for every}\ i\in I.$$
The {\em cuspidal block} is the full triangulated subcategory of cuspidal objects
$$D^b_{cusp}(X)\subset D^b(X).$$
\end{defn}

Philosophically, the cuspidal block captures information  about the derived category
not already encoded in $D^b(X_i)$ for $i\in I$.
We show in Theorem~\ref{prototype} that under quite general assumptions $D^b_{cusp}(X)$ is an admissible subcategory
and in fact the first block in the ``inclusion--exclusion'' semi-orthogonal decomposition of $D^b(X)$.
In our applications morphisms $\pi_i$ are forgetful maps such as $\M_{0,n}\to\M_{0,n-1}$
and thus an $S_n$-equivariant description of $D^b(X)$ can be reduced to an $S_n$-equivariant description of $D^b_{cusp}(X)$.

\begin{question}\label{cuspidal M}
Find a full, $S_n$-invariant, exceptional collection in the cuspidal block $D^b_{cusp}(\M_{0,n})$ with respect to all the forgetful maps 
$\M_{0,n}\to\M_{0,n-1}$.  
\end{question}

An answer to Question \ref{cuspidal M} together with Proposition \ref{semiortMN} (an application of Theorem \ref{prototype})
will therefore answer Question \ref{Orlov question}.
\begin{prop}\label{semiortMN}
We write $\bM_N\simeq\bM_{0,n}$ for the moduli space of stable rational curves with points marked by any $n$-element set $N$.
Then $D^b(\bM_N)$ admits a semi-orthogonal decomposition
\begin{equation}\label{lkehhagB}
D^b(\bM_N)=\langle D^b_{cusp}(\bM_N), \ \{\pi_K^*D^b_{cusp}(\bM_{N\setminus K})\}_{K\subset N},\ \cO\rangle
\end{equation}
where $K$ runs over subsets with $1\le |K|\le n-4$
in the order of increasing cardinality $|K|$ and $\pi_K: \M_N\ra \M_{N\setminus K}$ is the map that forgets markings in $K$. 
\end{prop}

We mention the answer to Question \ref{cuspidal M} in the first few small $n$ cases. 
\begin{ex} 
Let $T(-\log)$ be the rank $n-3$ vector bundle on $\M_{0,n}$ of vector fields tangent to its (normal crossing) boundary divisor.
It is easy to deduce from the results of \cite{KeelTevelev} that $T(-\log)$ is an exceptional vector bundle and an element of $D^b_{cusp}(\M_{0,n})$
for every $n$. This fact, which we view as a manifestation of rigidity of $\M_{0,n}$, 
was one of our original motivations for writing this paper.
For small $n$,  $D^b_{cusp}(\M_{0,n})$ has the following full $S_n$-equivariant exceptional collection:
\begin{itemize}
\item $(n=4)$.
$T(-\log)$ ($1$ object);
\item $(n=5)$.
$T(-\log)$ ($1$ object);
\item $(n=6)$.
$\cO_{\bP^1\times\bP^1}(-1,-1), \cL^\vee, T(-\log)$ ($12$ objects).
\end{itemize}
Here $\bP^1\times\bP^1\subset\M_{0,6}$ are boundary divisors of type $(3,3)$
and $\cL$ is a pull-back of the symmetric GIT polarization (the Segre cubic).
\end{ex}

\

We apply this approach to the Losev--Manin \cite{LM} moduli space. For an $n$-element set $N$, we let 
$\tN=\{0,\infty\}\sqcup N$.
We write  $\LM_N$ for the moduli space of nodal linear chains of $\bP^1$'s marked by $\tN$ 
with $0$ is on the left tail and $\infty$ is on the right tail of the chain. This is a ``simplified'' version of $\M_{0,n}$,
with linear chains replacing arbitrary trees.
The stability conditions are
\begin{itemize}
\item Marked points are never at the nodes.
\item Only points marked by $N$ are allowed to coincide with each other. 
\item Every $\bP^1$ has at least three special points (marked points or nodes).
\end{itemize}

The space $\LM_N$ has an action by 
the group $S_2\times S_{N}$ permuting markings. The action of~$S_2$, which we call the {\em Cremona action}, interchanges $\inf$ and $0$.
Both psi-classes  $\psi_0$ and $\psi_\infty$ induce birational morphisms $\LM_N\to\bP^{n-1}$, ''Kapranov models'',
which realize $\LM_N$ as an iterated blow-up of $\bP^{n-1}$ in $n$ points (standard basis vectors) followed
by blowing up $n\choose 2$ proper transforms of lines connecting points, etc.\footnote{We note that the other $\psi$-classes of $\LM_N$ are trivial.}.  
In these coordinates the Cremona action
is given by the standard Cremona involution
$$(x_1:\ldots:x_n)\to \left({1\over x_1}:\ldots:{1\over x_n}\right).$$

The Losev--Manin space $\LM_N$ is a toric variety of dimension $n-1$. Its toric orbits (or their closures, the boundary strata
of the moduli space) can be described as follows.
Every non-trivial bipartition $N=N_1\sqcup N_2$ corresponds to the boundary divisor, which we denote $\delta_{N_1}$,
parametrizing (degenerations of) chains of two~$\bP^1$, one with markings $N_1\cup\{0\}$,
and another with markings $N_2\cup\{\infty\}$.
This notation is different from the standard notation for~$\bM_{0,n}$ (where an analogous divisor
is denoted by $\delta_{N_1\cup\{0\}}$) but more convenient for us.
More generally, every partition $N=N_1\sqcup\ldots\sqcup N_k$ with $|N_i|>0$ for every~$i$
corresponds to the boundary stratum 
$$Z_{N_1,\ldots, N_k}=\de_{N_1}\cap \de_{N_1\cup N_2}\cap\ldots\cap
\de_{N_1\cup \ldots\cup N_{k-1}}$$ 
which parametrizes (degenerations of) linear chains of $\bP^1$'s with points marked by, respectively,
$N_1\cup\{0\}$, $N_2$,\ldots, $N_{k-1}$, $N_k\cup\{\infty\}$. 
We can identify
$$Z_{N_1,\ldots, N_k}\simeq \LM_{N_1}\times\ldots\times\LM_{N_k},$$ 
where the left node of every $\PP^1$ is marked by $0$
and the right node by $\infty$. 

We have a collection of forgetful maps
$$\pi_K:\, \LM_N\to \LM_{N\setminus K}$$
for every subset $K\subset N$ with $1\le|K|\le n-1$.
It is given by forgetting points marked by $K$ and stabilizing.
In particular, we can define the cuspidal block $D^b_{cusp}(\LM_N)$ and applying Theorem~\ref{prototype}, we show that 
we have a similar statement as for $\M_{0,n}$ (Prop.~\ref{semiortMN}):
\begin{prop}\label{semiort}
$D^b(\LM_N)$ admits the semi-orthogonal decomposition
$$
D^b(\LM_N)=\langle D^b_{cusp}(\LM_N), \ \{\pi_K^*D^b_{cusp}(\LM_{N\setminus K})\}_{K\subset N},\ \cO\rangle
$$
where subsets $K$ with $1\le |K|\le n-2$ are ordered 
by  increasing cardinality.
\end{prop}

Next we construct a collection $\hat\bG$ of sheaves on $\LM_N$. We note that in this definition, and in the rest of the paper, we do not always distinguish 
notationally between divisors and line bundles. 
\begin{defn}\label{torsion block}\label{defG}
Let $\bG_N=\{G_1^\vee,\ldots,G_{n-1}^\vee\}$ be the set of following line bundles on $\LM_N$:
$$G_a=a\psi_{0}\otimes\cO\big(-(a-1)\sum_{k\in N}\de_{k}-(a-2)\sum_{k,l\in N}\de_{kl}-\ldots-\sum_{J\subset N, |J|=a-1}\de_{J}\big)
$$
for every $a=1,\ldots,n-1$.
Let $\hat\bG$ be the collection of sheaves
$$\hat\bG=\bigcup_Z(i_Z)_*[\bG_{N_1}^\vee\boxtimes\ldots\boxtimes\bG_{N_t}^\vee]$$ 
on $\LM_N$ of the form
$$\cT=(i_Z)_*\cL,\quad \cL=G_{a_1}^\vee\boxtimes\ldots\boxtimes G_{a_t}^\vee$$
for all strata $Z=Z_{N_1,\ldots,N_t}$ such that $N_i\ge 2$ for every $i$ and for all $1\leq a_i\leq |N_i|-1$.
Here $i_Z:\,Z\hookrightarrow \LM_N$ is the inclusion map.
If $t=1$ we get line bundles $\bG_N$ and for $t\geq2$ these sheaves are torsion sheaves.
 \end{defn}

\begin{thm}\label{hhbhjghjgh}\label{mainLM}
$\hat\bG$ is a full exceptional  collection in $D^b_{cusp}(\LM_N)$,
which is
equivariant under the group $S_2\times S_N$.
The number of objects in $\hat\bG$ is equal to $!n$,
the~number of derangements of $n$ objects (permutation without fixed points).
\end{thm}

This is our main theorem, with proof occupying sections \ref{LosevManin} and \ref{fullness!!!}.
It gives a new curious formula for the number of derangements\footnote{We are unaware of a combinatorial ``bijective'' proof of this identity.}:
\begin{equation}
\sum_{k_1+\ldots+k_t=n\atop k_1,\ldots,k_t\ge 2}\left(
{n\atop k_1\ \ldots\ k_t}\right)(k_1-1)\ldots(k_t-1)=!n,
\label{derangedcount}
\end{equation}
where $\left({n\atop k_1\ \ldots\ k_t}\right)=\frac{n!}{k_1!\ldots k_t!}$.
As a corollary, we see that K-theory of $\LM_N$ is a permutation representation of $S_2\times S_n$
in a very concrete way, which should be contrasted with description of its equivariant 
Euler--Poincare polynomial as an alternating sum in \cite{BergstromMinabeLM}.
 
The ordering of $\hat\bG$ that tuns it into an exceptional collection is quite elaborate and discussed in \S\ref{LosevManin}.
The real difficulty though is to prove fullness, which is done in \S\ref{fullness!!!}. 
Note that fullness would follow at once if phantom subcategories (admissible subcategories with trivial $K$-group)
did not exist on smooth projective toric varieties.

\begin{rmk}
The line bundles $G_1,\ldots,G_{n-1}$ on $\LM_n$ may appear ad hoc, but in fact they have a very nice description in terms of the (minimal) wonderful compactification 
$\overline{\PGL_n}$ of $\PGL_n$
(which contains  $\LM_n$ as the closure of the maximal torus). Namely, they are precisely the restrictions of generators
of the nef cone of $\overline{\PGL_n}$, see Prop.~\ref{BrionIdea} for a more precise statement.
It~would be interesting to relate derived categories of $\overline{\PGL_n}$ and  $\LM_n$.
\end{rmk}

It is worth noting that we do not know any smooth projective 
toric varieties $X$ with an action of a finite group $\Gamma$ normalizing the torus action
which do not have a $\Gamma$-equivariant exceptional collection $\{E_i\}$ of maximal possible length (equal to the topological Euler characteristic of $X$).
Its existence would imply that the K-group $K_0(X)$ is a permutation $\Gamma$-module.
In the Galois setting (when $X$ is defined over a field which is not algebraically closed
and $\Gamma$ is the absolute Galois group), an analogous statement was conjectured by Merkurjev and Panin \cite{MP}.
Of course one may further wonder if $\{E_i\}$ is in fact full, which is related to the (non)-existence of phantom categories on $X$,
another difficult general open question. 

\medskip 

We refer to \cite{CT1,CT2,CT3} for the background information on birational geometry of $\M_{0,n}$, the Losev--Manin space
and other related spaces. We refer to \cite{Huy} for background on semi-orthogonal decompositions.

\subsection{Acknowledgements}
We are grateful to Alexander Kuznetsov for a  suggestion to think about the derived category of pointed curves in the equivariant setting
and for several improvements of the exposition. We are grateful to Michel Brion for pointing out the connection to the wonderful compactification of $\PGL_n$.
We thank Asher Auel, Chunyi Li, Daniel Halpern-Leistner, Emanuele Macr\`i and Dimitri Zvonkine 
for useful conversations. We are grateful to the referee for the careful reading and numerous very helpful comments. 

\smallskip 

The first named author was supported by NSF grants DMS-1529735 and DMS-1701752. The second named author was supported by NSF grants DMS-1303415 and DMS-1701704. Parts of this paper were written while the first author was visiting the Institut des Hautes \'Etudes Scientifiques in Bures-sur-Yvette, France
and the second author was visiting the Fields Institute in Toronto, Canada.


\section{An equivariant version of Orlov's blow-up theorem}\label{equiOrlovS}

Orlov's blow-up theorem \ref{OrlovI} is a categorification of the following fact.
Let $X$ be a smooth projective variety and let $Y\subseteq X$ be a smooth subvariety of codimension~$l$.
Let $\tilde X$ be the blow-up of $X$ along $Y$. We have decomposition of cohomology
with integral coefficients, see e.g.~\cite[Th.~7.31]{Voisin}
\begin{equation}\label{HomBlow}
H^*(\tilde X)\simeq\left[H^*(Y)\otimes H^+(\PP^{l-1})\right]\oplus H^*(X).
\end{equation}

Now consider the following more general situation. Let 
$Y_1,\ldots,Y_n\subseteq X$ be smooth transversal subvarieties of codimension $l$. For any subset $I\subseteq\{1,\ldots,n\}$, 
we denote by $Y_I$ the intersection $\cap_{i\in I}Y_i$. In particular, $Y_\emptyset=X$.
Let $q:\,\tX\to X$ be an iterated blow-up of (proper transforms of) $Y_1,\ldots,Y_n$.
Since the intersection is transversal, blow-ups can be done in any order.
The analogue of \eqref{HomBlow} in this situation was worked out in \cite[Prop. 6.1]{BergstromMinabe}:
\begin{equation}\label{HomItBlow}
H^*(\tilde X)\simeq\bigoplus_{I\subset\{1,\ldots,n\}\atop I\ne\emptyset}\left[H^*(Y_I)\otimes H^+(\PP^{l-1})^{\otimes|I|}\right]\oplus H^*(X),
\end{equation}
which we are going to rewrite as
$$
H^*(\tilde X)\simeq\bigoplus_{I\subset\{1,\ldots,n\}}\left[H^*(Y_I)\otimes H^+(\PP^{l-1})^{\otimes|I|}\right]
.$$
The analogue of Theorem \ref{OrlovI} is also straightforward. We fix the following notation.
Let $E_i$ be the exceptional divisor over $Y_i$ for every $i=1,\ldots,n$.
For~any subset $I\subseteq\{1,\ldots,n\}$, let 
$$E_I=q^{-1}(Y_I)=\cap_{i\in I}E_i.$$
In particular, $E_\emptyset=\tX$.
Let $i_I:\,E_I\hookrightarrow\tX$ be the inclusion.

\begin{lemma}\label{MammaMia}
Let  
$\{F_I^\beta\}$ be  a (full) exceptional collection in $D^b(Y_I)$ for every subset $I\subseteq\{1,\ldots,n\}$.
There exists a (full) exceptional collection in $D^b(\tilde X)$ with blocks
$$B_{I,J}=(i_I)_*\left[(Lq|_{E_I})^*(F_I^\beta)\left(\sum_{i=1}^n J_iE_i\right)\right]$$
for every subset $I\subseteq\{1,\ldots,n\}$ (including the empty set) and for every $n$-tuple of integers $J$
such that $J_i=0$ if $i\not\in I$ and $1\le J_i\le l-1$ for $i\in I$.

The blocks are ordered in any linear order which respects the following partial order:
$B_{I^1,J^1}$ precedes $B_{I^2,J^2}$ if $\sum\limits_{i=1}^n J^1_iE_i\ge \sum\limits_{i=1}^n J^2_iE_i$ (as effective divisors).
\end{lemma}

\begin{proof}
 We argue by induction on $n$, the case $n=1$ being Orlov's theorem.
 We decompose $q:\,\tX\to X$ as a blow-up $q_0:\,X'\to X$ of $Y_n$
 and an iterated blow-up $q':\,\tX\to X'$ of proper transforms $Y_1',\ldots,Y_{n-1}'$
 of $Y_1,\ldots,Y_{n-1}$. By~Orlov's theorem, $X'$ carries a (full) exceptional collection $E'^\alpha$, namely
$$
i'_*\left[(q_0|_E)^*(F_n^\beta)((l-1)E)\right],
\ldots,
i'_*\left[(q_0|_E)^*(F_n^\beta)(E)\right],
Lq_0^*(F^\beta_\emptyset).
$$
Here $i':\,E\hookrightarrow X'$ is the exceptional divisor
and $q_0|_E$ is a projective bundle.

More generally, for every subset $I'\subseteq\{1,\ldots,n-1\}$, let $Y'_{I'}=\cap_{i\in I'}Y'_i$ be the proper transform of $Y_{I'}$ isomorphic to the
blow-up of $Y_{I'}$ in $Y_{I'\cup\{n\}}$.
By~Orlov's theorem, $Y'_{I'}$ 
carries a (full) exceptional collection $F'^\beta_{I'}$, namely
$$
(i'_{I'})_*\left[(q_0|_{E_n^{I'}})^*(F_{I'\cup\{n\}}^\beta)((l-1)E)\right],
\ldots,\qquad\qquad\qquad\qquad$$
$$\qquad
(i'_{I'})_*\left[(q_0|_{E_n^{I'}})^*(F_{I'\cup\{n\}}^\beta)(E)\right],
L(q_0|_{Y'_{I'}})^*(F^\beta_{I'}).
$$
Here $i'_{I'}:\,E_n^{I'}\hookrightarrow Y'_{I'}$ is the exceptional divisor over $Y_{I'\cup\{n\}}$.

Applying the inductive assumption gives an exceptional collection on~$\tX$
with blocks
$$(i_{I'})_*\left[(Lq'|_{E_{I'}})^*(F'^\beta_{I'})\left(\sum_{i=1}^{n-1} J_iE_i\right)\right]$$
for every subset $I'\subseteq\{1,\ldots,n-1\}$ (including the empty set) and for every $(n-1)$-tuple of integers $J$
such that $J_i=0$ if $i\not\in I'$ and $1\le J_i\le l-1$ for $i\in I'$.

The blocks are ordered in any linear order which respects the following partial order:
$B_{I'^1,J^1}$ precedes $B_{I'^2,J^2}$ if $\sum\limits_{i=1}^{n-1} J^1_iE_i\ge \sum\limits_{i=1}^{n-1} J^2_iE_i$ (as effective divisors).
We have to check that these blocks are the same as in the statement of the lemma. It is clear that 
$$(Lq'|_{E_{I'}})^*(L(q_0|_{Y'_{I'}})^*(F^\beta_{I'}))
\simeq
(Lq|_{E_{I'}})^*(F_{I'}^\beta).
$$
This takes care of the last element in $F'^\beta_{I'}$. For the rest, we have to show that
$$(i_{I'})_*\left[(Lq'|_{E_{I'}})^*\left((i'_{I'})_*\left[(q_0|_{E_n^{I'}})^*(F_{I}^\beta)(J_nE)\right]\right)\left(\sum_{i=1}^{n-1} J_iE_i\right)\right]
\simeq$$
$$\qquad\qquad\qquad
(i_I)_*\left[(Lq|_{E_I})^*(F_I^\beta)\left(\sum_{i=1}^n J_iE_i\right)\right],$$
where $I=I'\cup\{n\}$. 
It suffices to show that
$$(Lq'|_{E_{I'}})^*\left((i'_{I'})_*\left[(q_0|_{E_n^{I'}})^*(F_{I}^\beta)(J_nE)\right]\right)\left(\sum_{i=1}^{n-1} J_iE_i\right)
\simeq$$
$$\qquad\qquad\qquad
(\phi)_*\left[(Lq|_{E_I})^*(F_I^\beta)\left(\sum_{i=1}^n J_iE_i\right)\right],$$
where $\phi:\,E_{I}\hookrightarrow E_{I'}$ is the inclusion.
Applying projection formula, this reduces to 
$$(Lq'|_{E_{I'}})^*\left((i'_{I'})_*\left[(q_0|_{E_n^{I'}})^*(F_{I}^\beta)\right]\right)
\simeq(\phi)_*\left[(Lq|_{E_I})^*(F_I^\beta)\right],$$
which follows by flat base change.

The last order of business is to prove the claim about the order of blocks.
We made a choice of blowing up $Y_n$ first, accordingly the collection has blocks $B_{I',J'}$
for every subset $I'\subseteq\{1,\ldots,n-1\}$ (including the empty set) and for every $(n-1)$-tuple of integers $J'$
such that $J'_i=0$ if $i\not\in I'$ and $1\le J'_i\le l-1$ for $i\in I'$.
The blocks are ordered in any linear order which respects the following partial order:
$B_{I'^1,J'^1}\prec B_{I'^2,J'^2}$ if $\sum\limits_{i=1}^{n-1} J'^1_iE_i> \sum\limits_{i=1}^{n-1} J'^2_iE_i$ (as effective divisors).
Each block $B_{I',J'}$ is a sequence of blocks $B_{I,J}$ from the statement of the Lemma, where 
$I\cap\{1,\ldots,n-1\}=I'$ and $J_i=J'_i$ for $i<n$. They are ordered in the decreasing order by $J_n$.
In particular, if $B_{I^1,J^1}$ precedes $B_{I^2,J^2}$ then either $\sum\limits_{i=1}^n J^1_iE_i-\sum\limits_{i=1}^n J^2_iE_i$ is an effective divisor
or $\sum\limits_{i=1}^{n-1} J^2_iE_i-\sum\limits_{i=1}^{n-1} J^1_iE_i$ is not effective.
Therefore, it suffices to prove that, for any two blocks $B_{I^1,J^1}$ and $B_{I^2,J^2}$,
if $\sum\limits_{i=1}^{n} J^1_iE_i-\sum\limits_{i=1}^{n} J^2_iE_i$ is not an effective divisor, then
$\{B_{I^1,J^1},  B_{I^2,J^2}\}$ form an exceptional sequence. 
If $\sum\limits_{i=1}^{n-1} J^1_iE_i-\sum\limits_{i=1}^{n-1} J^2_iE_i$ is not effective then we are done by the above.
But if it is effective, then $\sum\limits_{i=2}^{n} J^1_iE_i-\sum\limits_{i=2}^{n} J^2_iE_i$ is not effective,
and we are again done by the above (by changing the order of blow-ups and blowing up $Y_1$ first).
\end{proof}

\begin{rmk}
The same argument shows, more generally, 
 that even in the absence of exceptional collections, there exists a
 semi-orthogonal decomposition (s.o.d.)  of $D^b(\tilde X)$ with blocks
$$B_{I,J}=(i_I)_*\left[(Lq|_{E_I})^*(D^b(Y_I))\left(\sum_{i=1}^n J_iE_i\right)\right]$$
(the same notation and order as in the lemma).
We  stated the lemma for exceptional collections 
with an eye towards its equivariant version.
\end{rmk}

Continuing with the set-up of Lemma~\ref{MammaMia}, let $G$ be a finite group acting on $X$ permuting $Y_1,\ldots,Y_n$.
Then it also acts on $\tilde X$ and the morphism $q$ is $G$-equivariant. Let $G_I\subseteq G$ be the normalizer of $Y_I$ for each subset 
$I\subseteq \{1,\ldots,n\}$ (in particular, $G_\emptyset=G$). 

\begin{lemma}\label{BigDaddy}
Let  
$\{F_I^\beta\}$ be  a (full) $G_I$-equivariant exceptional collection in $D^b(Y_I)$ for every subset $I\subseteq\{1,\ldots,n\}$.
We assume that if $Y_I=gY_{I'}$ for some $g\in G$, then $\{F_I^\beta\}=g\{F_{I'}^\beta\}$.
There exists a (full) $G$-equivariant exceptional collection in $D^b(\tilde X)$ with blocks
$B_{I,J}$ (the same as in Lemma~\ref{MammaMia}).
\end{lemma}

\begin{proof}
It suffices to observe that the blocks $B_{I,J}$ are permuted by $G$.
\end{proof}

Next we recall a few facts and notation from \cite{BergstromMinabe} in order to prove Theorem~\ref{BMapproach}.
The subgroup $S_k \times S_{n-k}\subseteq S_n$ preserves the weight $\bold a$ and therefore acts on
$\M^n_{k,l}$.
 We have $(S_k \times S_{n-k})$-equivariant reduction morphisms:
\begin{equation}\label{HassettTower}
\M^n_{k,1}\to\M^n_{k,2}\to\ldots\to\M^n_{k,r(n,k)},
\end{equation}
where the first map is an isomorphism.
Each of the maps in \eqref{HassettTower} is an iterated blow-up of transversal loci of the same codimension
permuted by $S_k\times S_{n-k}$. 
Specifically, for every subset $I\subset\{k+1,\ldots,n\}$ of cardinality $l+1$, let $\M^n_{k,l+1}(I)\subseteq\M^n_{k,l+1}$
be the closure of the locus where points marked by $I$ collide.
The reduction morphism $\M^n_{k,l}\to\M^n_{k,l+1}$ is the blow-up along the transversal union $\cup_I\M^n_{k,l+1}(I)$
of sub varieties of codimension~$l$,
where $I$ runs over all subsets of $\{k+1,\ldots,n\}$ of cardinality $l+1$ \cite[Lemma~3.1]{BergstromMinabe}.
Intersections of these loci are described in \cite[Section 3.2]{BergstromMinabe} as follows.
Let $I_1,\ldots,I_m\subset \{k+1,\ldots,n\}$ are subsets of cardinality $l+1$. Then $\cap_{i=1}^m\M^n_{k,l+1}(I_i)\ne\emptyset$
if and only if the subsets $I_1,\ldots,I_m$ are disjoint. In this case the intersection
is isomorphic to $\M^{n-lm}_{k+m,l+1}$. Moreover, the stabilizer of this stratum in $S_k\times S_{n-k}$ acts on it through a subquotient
contained in $S_{k+m, n-lm-k-m}$. Applying Lemma~\ref{BigDaddy} proves Theorem~\ref{BMapproach}.



\section{The cuspidal block}\label{zfasfhasfh}

Recall that by Def. \ref{zxfbdfdfhdfd}, we call an object 
$E\in D^b(X)$ {\em cuspidal} with respect to a given collection of morphisms $\pi_i:\,X\to X_i$ ($i\in I$)
between smooth projective varieties, if 
$${R\pi_i}_*E=0\quad \hbox{\rm for every}\ i\in I.$$
The {\em cuspidal block} is the full triangulated subcategory of cuspidal objects
$$D^b_{cusp}(X)\subset D^b(X).$$

\begin{lemdef}\label{support}
In the set-up of Definition~\ref{zxfbdfdfhdfd},
the support of any cuspidal object is a union of irreducible 
closed subsets $Z\subset X$ such that
$$\dim\, \pi_i(Z)<\dim\, Z\quad \hbox{\rm for every}\ i\in S.$$
We call any subset $Z$ with this property (independently of whether they are the support of a cuspidal object or not) \emph{massive}.
Recall that the topological support of an object $E\in D^b(X)$ is the support of its cohomology sheaves. 
\end{lemdef}

\begin{proof}
Let $Z$ be the topological support of $E\in D^b(X)$. 
Suppose $Z$ contains an irreducible component $Z_0$
such that $\dim\, \pi_i(Z_0)=\dim\, Z_0$. 
We  denote $\pi:=\pi_i$ and $Y:=X_i$ as we won't need other maps and spaces.
By passing to an open subset of $Y$ and taking its preimage under $\pi$, we can assume that
$Z$ is a disjoint union of $Z_0$ and $Z_1$ (with $Z_1$ possibly empty and not necessarily irreducible). 
We may also assume that $\pi|_{Z_0}$ is finite.
It is well-known (\cite[Section~2]{Orlov formal}) that by changing $E$ to an isomorphic object, 
we may assume that $E$ is a bounded complex of sheaves supported on $Z$.
Thus $E=Ri_*\tilde E$, where $i:\,\tilde Z\hra X$ is an infinitesimal  thickening of $Z$ and $\tilde E\in D^b(\tilde Z)$.
Note that $\tilde Z$ is a disjoint union of subschemes $\tilde Z_0$ and $\tilde Z_1$ (with reduced subschemes 
$Z_0$ and $Z_1$). In~particular, $\tilde E=\tilde E_0\oplus \tilde E_1$, where $\tilde E_0$, resp., $\tilde E_1$, is a 
pull-back of $\tilde E$ to $Z_0$, resp.~$Z_1$. It follows that $R\tilde\pi_*(\tilde E_0)=0$ 
where $\tilde\pi=\pi\circ i$. Since $\tilde E_0\ne0$ and the map $\tilde\pi$ is affine, this gives a contradiction.
Indeed, if $\pi:\,X\to Y$ is an affine morphism of schemes then $R\pi_*E=0$ for some $E\in D_{QCoh}(\cO_X)$ implies that $E=0$,
see \cite[\href{http://stacks.math.columbia.edu/tag/0AVV}{Tag 0AVV}]{stacks-project}.
\end{proof}

We refer to the survey \cite{Kuznetsov Rat} for definitions and basic facts concerning semi-orthogonal decompositions in algebraic geometry.
The following is well-known, see e.g.~\cite[Lemma~2.4]{Kuznetsov Lefschetz}

\begin{prop}\label{asgasga}
Let $\pi:\,X\to Y$ be a morphism of smooth projective varieties such that $R\pi_*\cO_X=\cO_Y$.
Then $L\pi^*D^b(Y)$ is an admissible subcategory of $D^b(X)$ and 
there is a semi-orthogonal decomposition
$$D^b(X)=\langle D^b_{cusp}(X), L\pi^* D^b(Y)\rangle.$$
In particular,  $D^b_{cusp}(X)$ is an admissible subcategory. 
\end{prop}

Classical situations of this sort are provided 
by Orlov's theorems \cite{Orlov blow-up} on derived categories of a projective bundle and of a blow-up, which can be reformulated as follows:

\begin{thm}\label{OrlovI}
Let $\pi:\,X\to Y$ be a projective bundle of rank $r$ (with $Y$ a smooth projective variety).
Then $D^b_{cusp}(X)$ is an admissible subcategory of $D^b(X)$ and $D^b_{cusp}(X)$ has a semi-orthogonal decomposition
$$\langle 
\pi^*D^b(Y)\mathop{\otimes}\cO_\pi(-r),\ldots, 
\pi^*D^b(Y)\mathop{\otimes}\cO_\pi(-1)\rangle$$
\end{thm}

\begin{thm}\label{OrlovII}
Let $p:\,X\to Y$ be a blow-up of a smooth subvariety~$Z$ of codimension $r+1$ of 
a smooth projective variety $Y$. Let  $i:\,E\to X$ be the exceptional divisor and let $\pi=p|_Z$.
Then $D^b_{cusp}(X)$ is an admissible subcategory of $D^b(X)$ and has a semi-orthogonal decomposition
$$\left\langle 
{Ri}_*\left[\pi^*D^b(Z)\mathop{\otimes}\cO_\pi(-r)\right],\ldots, 
{Ri}_*\left[\pi^*D^b(Z)\mathop{\otimes}\cO_\pi(-1)\right]\right\rangle.$$
\end{thm}

In order to generalize Proposition~\ref{asgasga} to the set-up of several morphisms, we impose compatibility conditions. 
In subsequent sections we will consider several variants of moduli spaces of rational pointed curves, 
which will all fit into this framework.

\begin{thm}\label{prototype}
Let $\NN$ be the category of finite subsets of a fixed set with inclusions as morphisms.
Let $X$ be a contravariant functor from $\NN$
to the category of smooth projective varieties. For every $T\subseteq S$, we refer to the  morphism $X_S\to X_T$
as {\em forgetful map} and denote it by $\pi_{S\setminus T}$. 
We impose three assumptions: 
\begin{equation}\label{fbzfbsdb}
{R\pi_i}_*\cO_{X_S}=\cO_{X_{S\setminus \{i\}}}\quad \hbox{\rm for every}\ i\in S;
\end{equation} 
for all $i, k\in S$, $i\neq k$, the morphisms
\begin{equation}\label{a,sbgas}
\pi_i: X_{S\setminus \{k\}}\ra X_{S\setminus \{i, k\}}, \pi_k: X_{S\setminus\{i\}}\ra X_{S\setminus \{i,k\}} \hbox{ are $\Tor$-independent}
\end{equation} 
(as defined in \cite[Def. 36.21.2]{stacks-project})
and if we let
$$Y:=X_{S\setminus\{i\}}\times_{X_{S\setminus\{i, k\}}}X_{S\setminus\{k\}}$$
and $\alpha_{i,k}:\,X_S\to Y$ is the map induced by $\pi_i$ and~$\pi_k$, we have
\begin{equation}\label{asKJGBAMS}
{R\alpha_{i,k}}_*\cO_{X_S}=\cO_{Y}
\end{equation}
Under these assumptions we have a semi-orthogonal decomposition (s.o.d.)
$$
D^b(X_S)=\langle D^b_{cusp}(X_S), \quad \{L\pi_K^*D^b_{cusp}(X_{S\setminus K})\}_{K\subset S},\quad L\pi_S^*D^b(X_\emptyset)\rangle,
$$
where $K$ runs over proper subsets of $S$ in order of increasing cardinality. In~particular, 
$D^b_{cusp}(X_S)$ is an admissible subcategory of $D^b(X_{S})$.
\end{thm}

Following a suggestion of Kuznetsov,
we start with an abstract ``inclusion--exclusion'' principle in triangulated categories.
Perhaps we should remark that semi-orthogonal decompositions do not intersect well in general,
as a simple example of $D^b(\bP^1)=\langle \cO,\cO(1)\rangle=\langle \cO(2),\cO(3)\rangle$ shows.
However, we have the following.

\begin{lemma}\label{jmgfgk}
Let $\cT$ be a triangulated category with several s.o.d.
$$\cT = \langle A_1,B_1\rangle = \langle A_2,B_2\rangle = \ldots = \langle A_n,B_n\rangle.$$
Suppose the projection functors $\beta_i:\,\cT\to B_i$ (in the $i$-th decomposition)
have the property that, for every $j$,
$$\beta_i(A_j) \subset A_j,\quad \beta_i(B_j) \subset B_j.$$
Then we have a s.o.d.
$$\cT = \langle \cT_K \rangle_{K},\quad\hbox{\rm where}\quad \cT_K = (\cap_{i \not\in K} A_i) \cap (\cap_{i \in K}B_i)$$
and $K$ runs over  subsets of $\{1,\ldots,n\}$ in the order of increasing cardinality.
In~particular, $\cT_\emptyset=A_1\cap\ldots A_n$ is an admissible subcategory of $\cT$.
\end{lemma}

\begin{proof}
For all subsets $T\subseteq S:=\{1,\ldots,n\}$, we consider a full triangulated 
subcategory $A_T=\cap_{i\in T} A_i$.
We prove more generally that there is a semi-orthogonal decomposition
$A_T= \langle \cT_K \rangle$,
where $K$ runs over  subsets of $S$ containing $T$ in order of increasing cardinality.
The case $T=\emptyset$ is the statement in the theorem. 

We argue by induction on $n=|S|$ and by downwards induction on $|T|$ for a fixed $n$. 
If $n=1$ or $T=S$ then there is nothing to prove. Let $i\in S\setminus T$. Without loss of generality we assume $i=1$.

We claim that the semi-orthogonal decomposition
$\cT = \langle A_1,B_1\rangle$
 induces a semi-orthogonal decomposition
 \begin{equation}\label{ajshdbfa}
 A_T = \langle A_T\cap A_1, A_T\cap B_1\rangle.
 \end{equation}
 Indeed, the semi-orthogonality is obvious and moreover every object $X$ in $A_T$ fits into a distinguished triangle 
 $$\beta_1(X)\to X\to Y \to$$ with $Y\in A_1$.
 Since $\beta_1$ preserves $A_T$ by our assumptions, $\beta_1(X)\in A_T\cap B_1$.
 It follows that $Y\in A_T$ as well.

By the induction assumption, we have semi-orthogonal decompositions
$$A_T\cap A_1=A_{T\cup\{1\}}= \langle \cT_K \rangle,\quad\left(\hbox{\rm resp.}\quad A_T= \langle \cT'_{K'} \rangle\right),$$
where $K$  (resp.~$K'$) runs over  subsets of $S$ containing ${T\cup\{1\}}$
(resp. over  subsets of $S\setminus\{1\}$ containing $T$) and 
$$\cT'_{K'}= (\cap_{i \not\in K\cup\{1\}} A_i) \cap (\cap_{i \in K}B_i).$$
We claim that the semi-orthogonal decomposition $A_T= \langle \cT'_{K'}\rangle$
induces the semi-orthogonal decomposition
$$A_T\cap B_1 = \langle \cT'_{K'}\cap B_1 \rangle=\langle \cT_{K'\cup\{1\}} \rangle.$$
Indeed, the semi-orthogonality is clear. By definition of the semi-orthogonal decomposition,
for every object $X\in A_T\cap B_1$, we can write a sequence of morphisms (``filtration'')
$$0\to \ldots \to T_{K'_1}\to T_{K'_2}\to \ldots \to X\to 0,$$ such that every  morphism is included in the distinguished triangle
$$T_{K'_1}\to T_{K'_2}\to X_{K'_1}\to$$
with $X_{K'_1}\in \cT'_{K'_1}$.
Applying the functor $\beta_1$ to this sequence, and using our assumptions, gives a filtration of $X$
with subquotients $\beta_1(X_{K'_1})\in \cT'_{K'_1}\cap B_1$.

Combining these observations with \eqref{ajshdbfa}, we get a semi-orthogonal decomposition
$$
 A_T = \langle \cT_K,  \cT_{K'\cup\{1\}}\rangle,
$$
where $K$  (resp.~$K'$) runs over  subsets of $S$ containing ${T\cup\{1\}}$
(resp. over  subsets of $S\setminus\{1\}$ containing $T$) in order of increasing cardinality.

Finally, we have to show that we can reorder blocks to put them in the order of increasing cardinality.
If $|K_1|<|K_2|$ then choose an index $j\in K_2\setminus K_1$. Then $\cT_{K_1}\subset A_j$ and $\cT_{K_2}\subset B_j$.
Thus $\cT_{K_1}\subset \cT_{K_2}^\perp$.
\end{proof}

\begin{proof}[Proof of Theorem~\ref{prototype}]
By Prop. \ref{asgasga}, we have s.o.d.'s $D^b(X_S)= \langle A_i, B_i\rangle$, where 
$$A_i=\{E\in D^b(X_S)\,|\,R{\pi_i}_*E=0\},$$
$$B_i=L\pi_i^*(D^b(X_{S\setminus\{i\}})).$$
We now apply Lemma \ref{jmgfgk} to $\langle A_i, B_i\rangle$. 
The projection operators are
$$\beta_i = L\pi_i^* R\pi_{i*}.$$
Note that for all $i, k\in S$ with $i\neq k$ and all $E\in D^b(X_{S\setminus\{k\}})$ we have
\begin{equation}\label{Tor1}
R{\pi_i}_*L{\pi_k}^*E\simeq L{\pi_k}^*R{\pi_i}_*E,
\end{equation}
since by assumption, $\pi_i$ and $\pi_k$ are  $\Tor$-independent. 
This follows from \eqref{a,sbgas} combined with cohomology and base change: if, 
$\pi'_i$ and $\pi'_k$ are the projection maps from $Y=X_{S\setminus\{i\}}\times_{X_{S\setminus \{i, k\}}}X_{S\setminus\{k\}}$ and 
$\alpha:\,X_S\to Y$ is the canonical map, we have  
$${R\pi_i}_*L\pi_k^*E={R\pi'_i}_*R\alpha_*L\alpha^*{L\pi'_k}^*E={R\pi'_i}_*L{\pi'_k}^*E=L{\pi_k}^*R{\pi_i}_*E.$$
where the second equality is by the projection formula and \eqref{asKJGBAMS}. 
It follows that
$$R\pi_{j*} L\pi_i^* R\pi_{i*} = L\pi_i^* R\pi_{j*} R\pi_{i*} = L\pi_i^* R\pi_{i*} R\pi_{j*},$$
and in particular
$$\beta_i(A_j) \subset A_j.$$
Also, 
$$L\pi_i^* R\pi_{i*} L\pi_j^* = L\pi_i^* L\pi_j^* R\pi_{i*} = L\pi_j^* L\pi_i^* R\pi_{i*},$$
and thus 
$$\beta_i(B_j) \subset B_j.$$

It remains to show that, in the notation of Lemma~\ref{jmgfgk},
$$D^b(X_S)_K=L\pi_K^*D^b_{cusp}(X_{S\setminus K})$$
for every subset $K\subset T$. Equivalently,
\begin{equation}\label{ajshrgafhka}
\bigcap_{i\in K} B_i=L\pi_K^*D^b(X_{S\setminus K}).
\end{equation}
We can assume that $K=\{1,\ldots,k\}$. Then 
it follows from (\ref{Tor1}) that
$$\beta_1\circ\ldots\circ\beta_k=L\pi_K^*R\pi_{K*}.$$
Thus every object from the LHS of \eqref{ajshrgafhka} is isomorphic to an object from the RHS, and vice versa.
\end{proof}

\begin{ex}
Let $X_S=(\bP^1)^S$ with projections as forgetful maps.
Conditions \eqref{fbzfbsdb}, \eqref{a,sbgas} and \eqref{asKJGBAMS} are clearly satisfied.
$X_S$ is the only massive subset. Applying Theorem~\ref{OrlovI} successively, it follows that 
 $$D^b_{cusp}(X_S)=\langle \cO(-1,-1,\ldots,-1)\rangle,$$
 i.e.~every object in $D^b_{cusp}(X_S)$ is isomorphic to $\cO(-1,-1,\ldots,-1)\otimes_k K$,
where $K$ is a complex of vector spaces. Moreover, the semi-orthogonal decomposition
of Theorem~\ref{prototype} is induced by a standard exceptional collection of $2^{|S|}$ line bundles 
$\cO(n_1,\ldots,n_{|S|})$, where $n_i=0$ or $-1$ for every~$i$.

Note that this collection is obviously equivariant under the action of $\Aut(X_S)$, which is 
the semidirect product of $S_n$ and $(\PGL_2)^n$ for $n=|S|$. Various moduli spaces
considered in this paper can be viewed as ``compactified quotients'' of this basic example modulo $\bG_m$ or $\PGL_2$.
\end{ex}

\bp[Proof of Prop.~\ref{semiortMN}]
Recall that we need to prove that $D^b(\bM_N)$ admits a semi-orthogonal decomposition
\begin{equation}\label{lkehhagB}
D^b(\bM_N)=\langle D^b_{cusp}(\bM_N), \ \{\pi_K^*D^b_{cusp}(\bM_{N\setminus K})\}_{K\subset N},\ \cO\rangle
\end{equation}
where $K$ runs over subsets with $1\le |K|\le n-4$
in the order of increasing cardinality $|K|$. 
We apply Theorem~\ref{prototype}.
All conditions \eqref{fbzfbsdb}, \eqref{a,sbgas} and \eqref{asKJGBAMS} are satisfied. 
Recall that a simple criterion for $\Tor$-independence for maps $X\ra S$ and $T\ra S$ is that 
one of them is flat. Hence, condition \eqref{a,sbgas} holds as the forgetful maps $\pi_i:\M_{0,n}\ra\M_{0,n-1}$ are flat. 
The condition \eqref{asKJGBAMS} holds as the map is birational and the image has rational
singularities \cite[pg. 548]{Keel}.
\ep

Similarly, we have:  
\bp[Proof of Prop.~\ref{semiort}]
Recall that we need to prove that $D^b(\LM_N)$ admits the semi-orthogonal decomposition
$$
D^b(\LM_N)=\langle D^b_{cusp}(\LM_N), \ \{\pi_K^*D^b_{cusp}(\LM_{N\setminus K})\}_{K\subset N},\ \cO\rangle
$$
where subsets $K$ with $1\le |K|\le n-2$ are ordered 
by  increasing cardinality. 
We apply Theorem~\ref{prototype} to the forgetful maps 
$$\pi_i:\LM_N\ra\LM_{N\setminus\{i\}},\quad i\in N.$$ 
All conditions \eqref{fbzfbsdb}, \eqref{a,sbgas} and \eqref{asKJGBAMS} are satisfied. 
Note again that the forgetful maps $\pi_i$ for $i\in N$ are flat (they give the universal family); hence, condition \eqref{a,sbgas} holds. 
The condition  \eqref{asKJGBAMS} holds because 
$\alpha_{ij}$ is birational and $Y$ has toroidal, and therefore rational, singularities.
\ep


\section{The exceptional collection $\hat{\Bbb G}$ on the Losev--Manin space}\label{LosevManin}

\begin{prop}
An irreducible subset $Z\subset\LM_N$ is massive
if and only if $Z$ is a boundary stratum of the form $Z_{N_1,\ldots, N_t}$ with $|N_i|\ge 2$ for $i=1,\ldots,t$.
\end{prop}
 
\bp
Let $Z$ be a boundary stratum. If $N_i=\{a\}$ for some $i$ then $\pi_a$ restricted to $Z$ is one-to-one.
Hence $Z$ is not a massive subset. On the other hand, if $|N_i|\ge 2$ for every $i$ then $Z$ is  a massive subset. 
It remains to show that if $Z$ is a proper irreducible subset of
a boundary stratum which intersects its interior then $Z$ is not massive. But the interior of any stratum is an algebraic torus $\bG_m^r$
and projections onto coordinate subtori are realizable as forgetful maps. Thus $Z$ can not be massive.
\ep

\begin{prop}
The rank of the $K$-group of $D^b(\LM_n)$ (resp.~$D^b_{cusp}(\LM_n)$) is equal to  $n!$
(resp.~$!n$). 
\end{prop}

\bp
Since $\LM_N$ is a toric variety, its $K$-group is a free Abelian group and its topological Euler characteristic
(and thus the rank of its K-group) is equal to the number of torus fixed points,
which are clearly parametrized by permutations of~$N$.
The second part of the Proposition follows because both the rank of the K-group of $D^b_{cusp}(\LM_n)$
(by Prop.~\ref{semiort}) and $!n$ (by obvious reasons) satisfy the same recursion 
\begin{equation}\label{derangement}
n!=!n+\sum_{1\le k\le n-1}{n\choose k}!(n-k)+1
\end{equation}
Hence these numbers agree.
\ep

\bp[Proof of formula~\eqref{derangedcount}]
Recall that formula~\eqref{derangedcount} states that 
\begin{equation*}
\sum_{k_1+\ldots+k_t=n\atop k_1,\ldots,k_t\ge 2}\left(
{n\atop k_1\ \ldots\ k_t}\right)(k_1-1)\ldots(k_t-1)=!n,
\end{equation*}
where $\left({n\atop k_1\ \ldots\ k_t}\right)=\frac{n!}{k_1!\ldots k_t!}$.

We denote the left hand side by $d_n$ and set $d_0=1$, $d_1=0$. Let 
$$A=\sum_{n\ge 2}(n-1){x^n\over n!}=x^2\left({e^x-1\over x}\right)'=e^x(x-1)+1.$$
Then we have 
$$\sum_{m\ge 0}{d_m\over m!}x^m=1+A+A^2+A^3+\ldots={1\over 1-A}={e^{-x}\over 1-x}.$$
But \eqref{derangement} implies that 
$${1\over 1-x}=\left(\sum_{m\geq0} !m{x^m\over m!}\right)\left(\sum_{n\geq0} {x^n\over n!}\right),$$
(where we set $!0=0!=1$). Hence $d_n=!n$ and we are done.
\ep

\begin{lemma}\label{G}$\ $
\bi
\item[(1) ] Every $G_a$ is $S_{N}$-invariant and Cremona action takes it to $G_{n-a}$.
\item[(2) ] We have $G_1=\psi_0$ and $G_{n-1}=\psi_\inf$.
\item[(3) ] For every boundary divisor $\de=\de_{N_1}\simeq\LM_{N_1}\times\LM_{N_2}$, we have
$${G_a}_{|\de}=\begin{cases}
G_{a}\boxtimes \cO & \text{ if }\quad a<|N_1|\cr
\cO & \text{ if }\quad a=|N_1|\cr
\cO\boxtimes G_{a-|N_1|}& \text{ if }\quad a>|N_1|.\end{cases}$$
\ei
\end{lemma}

\begin{proof} Direct calculation.
\end{proof}

\begin{notn}
For an object $F\in D^b(X)$, we often use notation $R\Gamma(F)$ instead of $R\Gamma(X,F)$
when the space $X$ is clear from the context. 
\end{notn}

\begin{lemma}\label{-G}\label{push forward zero} $\ $
\bi
\item[(1) ] Every $G_a$ is nef (and hence globally generated), of relative degree $1$ with respect to any forgetful map $\pi_i$, $i\in N$.
\item[(2) ] $G_a^\vee\in D^b_{cusp}(\LM_n)$. In particular, each $G_a^\vee$ is acyclic. 
\item[(3) ] $R\Hom(G_a,G_b)=0$ if $a\ne b$.
\item[(4) ] $R\Gamma(-\psi_0+G_a-G_b)=R\Gamma(-\psi_\inf+G_b-G_a)=0$ if $a<b$.
\ei
In particular, $\bG_N$ is an $S_2\times S_n$ equivariant exceptional (in fact pairwise orthogonal) collection in $D^b_{cusp}(\LM_N)$ 
of $n-1$ line bundles.
\end{lemma}

\bp
Since $\LM_n$ is a toric variety, (1) will follow if $G_a$ is non-negative on toric boundary curves. 
This follows from Lemma~\ref{G}~(3) by induction on dimension. 
Since restriction of $G_a^\vee$ to each fiber of $\pi_i$ has vanishing cohomology, (2) follows by cohomology and base change.
Since $R\Hom(G_a,G_b)=R\Gamma(-G_a+G_b)$ and we can assume $a>b$ (by applying Cremona action), 
both (3) and (4) follow from Lemma \ref{easy acyclic}.
\ep


\begin{lemma}\label{easy acyclic}
Consider the divisor $$D=-dH+\sum m_IE_I$$
on $\M_{0,n}$ or $\LM_N$
written in some Kapranov model. The divisor $D$ is acyclic if 
$$1\leq d\leq n-3, \quad 0\leq m_I\leq n-3-\#I.$$
\end{lemma}
\bp
By consecutively restricting to boundary divisors $E_I$ starting with those with the largest $\#I$ and continuing to those with smaller $\#I$, note that all the restrictions are acyclic, hence $D$ has the same cohomology as $-dH$. Clearly,  $-dH$ is acyclic if and only if 
$1\leq d\leq n-3$. 
\ep

\begin{lemma}
$\hat\bG$ is a collection of $!n$  sheaves in $D^b_{cusp}(\LM_N)$.
\end{lemma}

\bp
Follows from Lemma~\ref{derangedcount} and Lemma~\ref{-G} (2).
\ep

It is worth mentioning that if $i:\,Z\hookrightarrow X$ is a closed embedding
of smooth projective varieties and $Z\ne X$ then 
the functor $Ri_*:\,D^b(Z)\to D^b(X)$ is in general not fully faithful.
Therefore, even though all sheaves in $\hat\bG$ are clearly exceptional 
in the derived category of their respective  support (being line bundles on a rational variety), we still have to prove:

\begin{lemma}\label{exceptional}
All sheaves in $\hat\bG$ are exceptional.
\end{lemma}

\bp
All sheaves in $\hat\bG$ are of the form $i_*i^*\cL=Ri_*Li^*\cL$, where $\cL$ is an invertible sheaf on $\LM_N$
and $i$ is an embedding of some massive stratum~$Z$.
We have
$$R\Hom(Ri_*Li^*\cL,Ri_*Li^*\cL)=R\Hom(\cL\mathop{\otimes}^L Ri_*\cO_Z,\cL\mathop{\otimes}^L Ri_*\cO_Z)=$$
$$=R\Hom(Ri_*\cO_Z,Ri_*\cO_Z).$$
So it suffices to prove that $Ri_*\cO_Z=i_*\cO_Z$ is an exceptional object.
Let  $Z$ be the intersection of boundary divisors $D_1,\ldots,D_s$.
Resolving $i_*\cO_Z$ by the Koszul complex
$$\ldots \ra\oplus_{1\leq i<j\leq s}\cO(-D_i-D_j)\ra  \oplus_{1\leq i\leq s}\cO(-D_i)\ra\cO\ra i_*\cO_Z\ra0,$$
we see that it suffices to prove that 
$$R\Gamma(\cO_Z(D_{i_1}+\ldots+D_{i_k}))=0$$
for every $1\le i_1<\ldots<i_k\le s$.
Using that $\cO_Z(D_i)$ has the form
$$
\cO\boxtimes\ldots\boxtimes\cO\boxtimes\cO(-\psi_\infty)\boxtimes\cO(-\psi_0)\boxtimes\cO\ldots\boxtimes\cO,
$$
we conclude that this is indeed the case.
\ep

\begin{lemma}\label{order}
$\hat\bG$ is an exceptional collection with respect to the following  order.
Let $\cT, \cT'\in \hat\bG$. Let $(k_1,\ldots,k_t;a_1,\ldots,a_t)$ and $(k'_1,\ldots,k'_s;a'_1,\ldots,a'_s)$ be the corresponding data.
Then $\cT>\cT'$ if the sequence $(a_1,-k_1,a_2,-k_2,,\ldots)$ is lexicographically (=alphabetically) larger than 
$(a_1',-k'_1,a'_2,-k'_2,\ldots)$.
\end{lemma}

\bp
Let $Z$ and $Z'$ be massive strata supporting sheaves $\cT>\cT'$ in $\hat\bG$. 
These sheaves have the form
${Ri_Z}_*\cL$ and 
${Ri_{Z'}}_*\cL'$, respectively.
We~have to show that $R\Hom(\cT,\cT')=0$. 
Let $U$ be the smallest stratum containing both $Z$ and $Z'$.
Then $U$ is the intersection of boundary divisors $D_1,\ldots,D_s$ (these divisors
are precisely the divisors containing both $Z$ and $Z'$). We have
$$R\Hom({Ri_Z}_*\cL,{Ri_{Z'}}_*\cL')=R\Hom({Lj_{Z'}}^*{Rj_Z}_*\cL,\cL').$$
By \cite[Cor. 11.4(i)]{Huy}, 
it suffices to prove that 
$$R\Hom({Rj_Z}_*\cL,{Rj_{Z'}}_*\cL'(D))=0$$
for every $D=D_{i_1}+\ldots+D_{i_k}$ with $1\le i_1<\ldots<i_k\le s$,
where $j_Z$ (resp.~$j_{Z'}$) denotes the embedding of $Z$ (resp., ~$Z'$) into $U$.
Let $W=Z\cap Z'$. We can assume that $W$ is non-empty as otherwise there is nothing to prove.
Let $i_{W,Z}:\,W\hookrightarrow Z$ and $i_{W,Z'}:\,W\hookrightarrow Z'$
be the inclusions. We note that $Z$ and $Z'$ intersect transversally along $W$ in $U$, and therefore $j_Z$ and $j_{Z'}$ are $\Tor$-independent.
Next we apply cohomology and base change:
$$R\Hom({Rj_Z}_*\cL,{Rj_{Z'}}_*\cL'(D))=
R\Hom({Lj_{Z'}}^*{Rj_Z}_*\cL,\cL'(D))=$$
$$
R\Hom({Ri_{W,Z'}}_*Li_{W,Z}^*\cL,\cL'(D))=
R\Hom(Li_{W,Z}^*\cL,Li_{W,Z'}^!\cL'(D)),
$$
where for some morphism $f:X\ra Y$, we denote $Lf^!(-)$ the adjoint functor to $Rf_*(-)$. By Grothendieck duality, we have for $E\in D^b(Y)$ that 
$Lf^!(E)=Lf^*(E)\otimes\om_f[\dim(f)]$. Here, $\om_f=\omega_X\otimes f^*\omega_Y^*$, $\dim(f)=\dim(X)-\dim(Y)$. 
So it suffices to prove that 
$$R\Hom(Li_{W,Z}^*\cL,Li_{W,Z'}^*\cL'\otimes(D+c_1(\cN)))=0,$$
where~$c_1(\cN)$ is the first Chern class of the normal bundle $\cN:=\cN_{W,Z'}$, i.e., 
the sum of all boundary divisors that cut out $W$ inside $Z'$, or alternatively, 
the sum of boundary divisors that cut out $Z$ but don't contain $Z'$.

Now we proceed case by case. We write $$W=\LM_{K_1}\times\LM_{K_2}\times\ldots,$$
$$R\Hom({Li_{W,Z}}^*\cL, {Li_{W,Z'}}^*\cL'(D+N))=C_1\boxtimes C_2\boxtimes \ldots,$$
where $C_1$ is computed on $\LM_{K_1}$, etc. Note that if  $N=N_1\sqcup\ldots\sqcup N_t$, resp. $N=N'_1\sqcup\ldots\sqcup N'_{t'}$ are the two partitions corresponding to $\cT$, resp. $\cT'$ (hence, $|N_i|=k_i$ and  $|N'_i|=k'_i$ for all $i$), then $W\neq\emptyset$ implies that either $N_1\subseteq N'_1$ or 
$N'_1\subseteq N_1$. In particular, we have that $|K_1|=\min(k_1,k_1')$ and if $k_1=k'_1$ , then we have $N_1=N'_1$. 

{\bf Case 1.} Suppose  $a_1>a_1'$.
We would like to show that $C_1=0$.

If $k_1<k_1'$, then $C_1=R\Hom(-G_{a_1}, -G_{a_1'}-\psi_\inf)$, where $-\psi_\inf$ is a contribution from $N$ (there is no contribution to $C_1$ from $D$). 
Hence, $C_1=0$ by Lemma~\ref{-G}~(4).

If $k_1=k_1'$, then there is no contribution from $c_1(\cN)$ to $C_1$ and we have that 
either $C_1=R\Hom(-G_{a_1}, -G_{a_1'})=0$ (if $D$ doesn't include~$D_{K_1}$)
or $C_1=R\Hom(-G_{a_1}, -G_{a_1'}-\psi_\inf)=0$
(if $D$ includes $D_{K_1}$).

Finally, if $k_1>k_1'$ then there are no contributions from $c_1(\cN)$ or $D$ to $C_1$ and 
 $C_1=R\Hom(L, -G_{a_1'})=0$, where $L=-G_{a_1}$ if $a_1<k_1'$ or $L=\cO$ otherwise. In both cases, $C_1=0$ by Lemma~\ref{-G}.

{\bf Case 2.} Suppose  $a_1=a_1'$, $k_1<k_1'$. As in Case 1, we have that 
$C_1=R\Hom(-G_{a_1}, -G_{a_1}-\psi_\inf)=0$.

{\bf Case 3.} Suppose  $a_1=a_1'$, $k_1=k_1'$, $D$ includes $D_{K_1}$.

In this case also $C_1=R\Hom(-G_{a_1}, -G_{a_1}-\psi_\inf)=0$.

{\bf Case 4.} Suppose  $a_1=a_1'$, $k_1=k_1'$, $D$ does not include $D_{K_1}$.

In this case $C_1=R\Hom(-G_{a_1}, -G_{a_1})=\CC$ is useless.
However, we can now proceed exactly as above restricting to the next Losev--Manin factor $\LM_{K_2}$ in $W$. Note that in general, the factors $\LM_{K_i}$ appearing in $W$ need not be positive dimensional, but in this case, since $K_1=K'_1$, we must have that $|K_2|\geq2$ and we can proceed by induction. 
The Lemma follows.
\ep

The Cremona action gives another possible linear order:
\begin{cor}\label{order2}
$\hat\bG$ is an exceptional collection with respect to the order $<'$:
$$(k_1,\ldots,k_t;a_1,\ldots,a_t)>'(k'_1,\ldots,k'_t;a'_1,\ldots,a'_t)$$
if the sequence
$(k_t-a_t,-k_t,k_{t-1}-a_{t-1},-k_{t-1},,\ldots)$ is lexicographically larger than 
$(k'_s-a_s',-k'_s,k'_{s-1}-a'_{s-1},-k'_{s-1},\ldots)$.
\end{cor}

\begin{rmk}
The linear orders $<$ and  $<'$  are clearly not $S_2\times S_N$ equivariant. The~lemma shows that
both orders refine the $S_2\times S_N$ equivariant relation $\prec$ given by paths in the quiver with arrows
$$\cT\to \cT'\qquad \Leftrightarrow\qquad R\Hom(\cT,\cT')\ne0.$$
In other words, this quiver has no cycles. It would be nice to describe it combinatorially.
It would be even better to explicitly describe the algebra
$$
\bigoplus_{\cT\prec \cT'}R\Hom(\cT,\cT').
$$
Here are some easy observations about the quiver:
\begin{enumerate}
\item If there is an arrow between $\cT$ and $\cT'$
then the corresponding strata have a non-empty intersection.
\item The line bundles can be arranged to be at the right of torsion sheaves in the collection: for any torsion sheaf $\cT'$ in $\hat\bG$ and any line bundle
$\cT=G_a^\vee$ we have (in the notations of the proof of Lemma \ref{order})
$$R\Hom(\cT,\cT')=R\Ga({G_a}_{|Z'}\otimes\cT')=C_1\boxtimes C_2\boxtimes\ldots,$$
and $C_1=R\Hom(L, G_{a'_1}^\vee)$, where $L=G_{a_1}$ if $a_1<k'_1$ and $L=\cO$ otherwise. In both cases $C_1=0$ by Lemma \ref{-G}. 

\item It is not true in general that sheaves can be pre-ordered by codimension of support.
For example, on $\LM_8$, the sheaf $\cT'$ with data $(3,2,3;2,1,1)$ and support $Z'$ has to be to the right of the sheaf $\cT$ with data $(3,5;1,3)$ and support $Z$ such that $Z'\subseteq Z$, as an easy computation as above shows that $R\Hom(\cT,\cT')\neq0$.
\end{enumerate}
\end{rmk}

We give more information about the quiver. We introduce the following terminology:
\begin{defn} 
Let $\cT\in \hat\bG$ with support $Z$:
$$Z=\LM_{K_1}\times\LM_{K_2}\times\ldots\times\LM_{K_t},\quad \cT=G^\vee_{a_1}\boxtimes\ldots\boxtimes G^\vee_{a_t}.$$
\begin{enumerate}
\item We call $\LM_{K_1}$ the first component of $Z$, $\LM_{K_2}$ the second component of $Z$, etc, $\LM_{K_t}$ the last component of $Z$. 

\item We say that we remove the component $\LM_{K_i}$ from $\cT$ if we consider the sheaf $\tilde{\cT}$ given by 
$$\tilde{Z}=\LM_{K_1}\times\ldots\times\LM_{K_{i-1}}\times \LM_{K_{i+1}}\times\ldots\times\LM_{K_t},$$
$$\tilde{\cT}=G^\vee_{a_1}\boxtimes\ldots\boxtimes G^\vee_{a_{i-1}}\boxtimes G^\vee_{a_{i+1}}\boxtimes\ldots\boxtimes G^\vee_{a_t}.$$

\item We say that the \emph{end data} of $\cT$ is $(k_1, k_t;k_1-a_1,a_t)$.
Clearly, different objects in $\hat\bG$ could have the same end data. 

\item Recall from the proof of Lemma \ref{order} that to show $R\Hom(\cT,\cT')=0$
it suffices to show that
$$R\Hom(Li_{W,Z}^*\cT,Li_{W,Z'}^*\cT'\otimes (N+D))=0,$$
where $W=Z\cap Z'$, ~$N$ is the first Chern class of the normal bundle $\cN_{W,Z'}$, i.e., 
the sum of boundary divisors that cut out $Z$, and $D=D_{i_1}+\ldots+D_{i_r}$ is a (possibly empty) 
sum of boundary divisors containing both $Z$ and $Z'$. We let
$$W=\LM_{S_1}\times\LM_{S_2}\times\ldots,$$
$$R\Hom({i_{W,Z}}^*\cT, {i_{W,Z'}}^*\cT'(N+D))=C_{S_1}\boxtimes C_{S_2}\boxtimes \ldots.$$
In what follows we will  refer to $C_{S_i}$ as the 
 \emph{components} of $R\Hom(\cT,\cT')$. 

\end{enumerate}
\end{defn}

Lemma \ref{order} and Lemma \ref{order2} have the following corollary, which can be used as an algorithm to determine, given a pair 
of torsion objects $\cT, \cT'$ in $\hat\GG$, whether $R\Hom(\cT,\cT')=0$ or $R\Hom(\cT',\cT)=0$. 
\begin{cor}\label{partial order}
Let $\cT, \cT'$ be torsion sheaves in $\hat\bG$ with supports $Z$, $Z'$ 
and end data $(k_1, k_t;b_1,b_t)$ and $(k'_1, k'_s;b'_1,b'_s)$. 
If the following inequalities both hold 
$$b_1+b_t\leq b'_1+b'_s\quad k_1+k_t-b_1-b_t\geq k'_1+k'_s-b'_1-b'_s$$
and one of them is a strict inequality, then $R\Hom(\cT,\cT')=0$. 
Moreover, if both inequalities are equalities, then $R\Hom(\cT,\cT')\neq0$ only possibly when  
$$b_1=b'_1,\quad b_t=b'_s,\quad  k_1=k'_1,\quad k_t=k'_s,$$
and the first and last components components are the same, i.e., 
$K_1=K'_1$ and $K_t=K'_s$. 
Whenever all these conditions hold, we have that
$$R\Hom(\cT,\cT')=0\quad\text{ if }\quad R\Hom(\tilde{\cT},\tilde{\cT'})=0,$$ 
where  $\tilde{\cT}$ (respectively $\tilde{\cT'}$), are the sheaves obtained from $\cT$ (respectively $\cT'$) after removing the 
first and last components  $\LM_{K_1}$ and $\LM_{K_t}$.  
\end{cor}

\bp
Recall that we have 
$$a_1=k_1-b_1,\quad a'_1=k'_1-b'_1,\quad a_t=b_t,\quad a'_s=b'_s.$$
By Lemma \ref{order}, if $k_1-b_1>k'_1-b'_1$, then $R\Hom(\cT,\cT')=0$. Similarly, by Lemma \ref{order2}, 
if $k_t-b_t>k'_s-b'_s$ then $R\Hom(\cT,\cT')=0$. Since we assume
$$(k_1-b_1)+(k_t-b_t)\geq (k'_1-b'_1)+(k'_s-b'_s),$$
it follows that we must have $k_1-b_1=k'_1-b'_1$, $k_t-b_t=k'_s-b'_s$. Now if $-k_1>-k'_1$, 
again by Lemma \ref{order}, we have $R\Hom(\cT,\cT')=0$. Similarly, if $-k_t>-k'_s$, 
by Lemma \ref{order2}, we have $R\Hom(\cT,\cT')=0$. Hence, we may assume that 
$k_1\geq k'_1$, $k_t\geq k'_s$. But then $(k_1+k'_1)-(k_t+k'_s)\geq0$, while
$$(k_1+k'_1)-(k_t+k'_s)=(b_1+b_t)-(b'_1+b'_s)\geq0.$$
Hence, we must have $k_1=k'_1$, $k_t=k'_s$, and hence, $b_1=b'_1$, $b_t=b'_s$.

If these equalities hold, for the intersection $Z\cap Z'$ to be non-empty, 
we must have that the first and last components components are the same, i.e., 
$K_1=K'_1$ and $K_t=K'_s$. As in the proof of Lemma  \ref{order} (Case 4), 
we can remove the first and last components  $\LM_{K_1}$ and $\LM_{K_t}$, from $Z$ and $Z'$ 
and proceed with the rest. 
\ep

We finish this section by relating line bundles $G_1,\ldots,G_{n-1}$ on $\LM_n$ to the wonderful compactification of $\PGL_n$.
Following \cite{Brion}, we identify $\Pic\,\overline{\PGL_n}$ with the weight lattice of $\PGL_n$.
Let $\alpha_1,\ldots,\alpha_{n-1}$ be simple roots and let $\omega_1,\ldots,\omega_{n-1}$ be fundamental co-weights.
It is shown in \cite{Brion} that $\alpha_1,\ldots,\alpha_{n-1}$ (resp.~$\omega_1,\ldots,\omega_{n-1}$)
span the effective cone (resp.~the nef cone) of $\overline{\PGL_n}$.
We identify $\LM_n$ with the closure of the maximal torus in $\PGL_n$.

\begin{prop}\label{BrionIdea}
Divisors on $\overline{\PGL_n}$ corresponding to  $\omega_1,\ldots,\omega_{n-1}$
 restrict to divisors $G_1,\ldots,G_{n-1}$ on the Losev--Manin space  $\LM_n$.
 \end{prop}

\begin{proof}
First we consider divisors $D_1,\ldots,D_n$ on $\overline{\PGL_n}$ which correspond to simple roots $\alpha_1,\ldots,\alpha_{n-1}$.
We will show that they restrict to total boundary divisors $$\Delta_1=\sum\delta_{0i},\quad \ldots,\quad \Delta_{n-1}=\sum\delta_{0i_1\ldots i_{n-1}}$$ 
 on the Losev--Manin space  $\LM_n$.
 Indeed, it is known (see e.g.~\cite{Brion}) that $D_1,\ldots D_{n-1}$ are the boundary divisors of $\overline{\PGL_n}$, i.e.
 $$\overline{\PGL_n}\setminus \PGL_n=D_1\cup \ldots \cup D_{n-1}.$$
 In particular, every $D_i$ restricts to a linear combination of boundary divisors of $\LM_n$.
 Since each of these divisors is  $\PGL_n$-invariant (acting by conjugation), the restriction is invariant under $S_n$ (=Weyl group),
 i.e.~it is a linear combination of total boundary divisors. 
 
$\overline{\PGL_n}$ is a spherical homogeneous space for the group $\PGL_n\times \PGL_n$ extending its action on $\PGL_n$ by left and right translations.
 The group of semi-invariant functions $\Lambda=k(\PGL_n)^{(\cB)}/k^*$ (where $\cB=B^-\times B$ is a Borel subgroup of $\PGL_n\times \PGL_n$)
 is identified with the root lattice of $\PGL_n$, which in turn is identified with the lattice of characters $M=k(T)^{(T)}/k^*$ of the maximal torus in $\PGL_n$, via restriction of $\cB$-semi-invariant functions.
 Every boundary divisor $D_i$ determines the functional $\rho(D_i):\,\Lambda\to\ZZ$ 
 (and so an element of the dual weight lattice $\Lambda^*=\ZZ^n/\langle 1,\ldots,1\rangle$)
 by taking a divisorial valuation of a function in $\Lambda$ along $D_i$. In~fact we have $\rho(D_i)=\omega_i=e_1+\ldots+e_i\mod \langle 1,\ldots,1\rangle$,
 the fundamental coweight (see e.g.~\cite{Brion}). These vectors span the  Weyl chamber and the fan of $\LM_n$ (as a toric variety) is precisely the fan of its Weyl group translates. 
 Moreover, vectors $\omega_i$ are primitive vectors along the rays which give boundary divisors $\delta_{0,1,\ldots,i}$, see \cite{LM}. So we are done by \cite[Lemma 6.1.6]{BK}.
 
 By pulling back a coordinate hyperplane in $\bP^{n-1}$ and symmetrizing, we get the formula
 $$\psi_0={n-1\over n}\Delta_1+\ldots+{1\over n}\Delta_{n-1}.$$
 Combining it with Definition~\ref{defG} yields the following formula for $G_i$'s:
 $$G_i=\sum_{j=1}^{n-1}B_{ij}\Delta_j,$$
 where $B_{ij}=i{n-j\over n}$ if $i\le j$ and $B_{ij}=B_{ji}$ if $i>j$.
It is well-known and easy to prove that the inverse of the matrix $B$ is the Cartan matrix of the root system $A_{n-1}$, and therefore
 $$\psi_i=\sum_{j=1}^{n-1}B_{ij}\alpha_j,$$
which finishes the proof.
\end{proof}


\section{Fullness of the exceptional collection $\hat{\Bbb G}$}\label{fullness!!!}

We will need the following more general set-up. 
\begin{defn} For every integer $r\ge -1$, define a contravariant functor $X^r$ from $\NN$
to the category of smooth projective varieties as follows. Let $X^r_N$ is an iterated blow-up of $\bP^{n+r}$ in $n$ general points (marked by $N$)
followed by the blow-up of the $n\choose 2$ proper transforms of the lines passing through two points, the $n\choose 3$ proper transforms of the planes 
passing through three points, etc. For example, 
$$X^{-1}_N=\LM_N, \quad X^{r}_\emptyset=\bP^r.$$
For every $M\subseteq N$, the forgetful morphism $\pi_{N\setminus M}: X^r_N\to X^r_M$
is induced by a linear projection from points in $N\setminus M$.

For every subset $S\subseteq N$ of cardinality at most $n+r-1$, we denote by $E_S\subseteq X^r_N$ the exceptional divisor over a  subspace
spanned by points in $S$.
\end{defn}

\begin{prop}
All conditions of Theorem~\ref{prototype} are satisfied. Thus we have a semi-orthogonal decomposition 
$$
D^b(X^r_S)=\langle D^b_{cusp}(X^r_S), \quad \{L\pi_K^*D^b_{cusp}(X^r_{S\setminus K})\}_{K\subset S},\quad L\pi_N^*D^b(\PP^r)\rangle,
$$
where $K$ runs over proper subsets of $S$ in the order of increasing cardinality. 
\end{prop}

\begin{notn}
For every $i\in N$, we have a birational morphism
$$f_i:\,X_N^r\to X^{r+1}_{N\setminus i},$$
obtained by blowing down exceptional divisors $E_S$, $i\in S$,
in the order of decreasing cardinality. 
\end{notn}

\begin{defn} ({\bf Strata in $X^r_N$.})
Consider partitions $$N=N_1\sqcup\ldots\sqcup N_k,\quad |N_u|>0  \quad (u=1\ldots, k-1).$$  

Denote
$$Z_{N_1,\ldots, N_k}=E_{N_1}\cap E_{N_1\cup N_2}\cap\ldots\cap E_{N_1\cup \ldots\cup N_{k-1}}.$$ 
We call  $Z_{N_1,\ldots, N_k}$ a stratum in $X^r_N$. We call a stratum in  $X^r_N$ to be massive, if it is 
the image of a massive stratum in $\LM_{n+r+1}$ via the the birational map $\LM_{n+r+1}\ra X^r_N$ 
which is the composition of the maps $f_i$, for those $i\notin N$.  
\end{defn}

For $r\geq0$, each stratum $Z_{N_1,\ldots, N_k}$ is the image of a stratum in $\LM_{n+r+1}$. 
For all $r\geq-1$, we can identify 
$$Z_{N_1,\ldots, N_k}\simeq \LM_{N_1}\times\ldots\times\LM_{N_{k-1}}\times X^r_{N_k},$$ 
where $X^r_{N_k}$ is the blow-up of $\PP^{|N_k|+r}$ at the linear subspaces spanned by the points in $N_k$.
If $r\geq0$, a stratum $Z_{N_1,\ldots, N_k}$ is massive if and only if $|N_u|\geq2$ for all $u=1,\ldots, k-1$ and $|N_k|+r>0$.

\begin{defn}
We let $\hat\bG_N^r$ be a collection of objects in $D^b(X_N^r)$ defined inductively
as follows: 
$$\hat\bG_N^{-1}:=\hat\bG_N,\qquad \hat\bG_N^{r+1}={Rf_i}_*(\hat\bG_N^{r}).$$
\end{defn}

\begin{defn}\label{defG2}
Consider the following line bundles on $X^r_N$ ($r\geq-1$):
$${G_a^r}^\vee=-aH+(a-1)\sum_{i        \in N} E_i+(a-2)\sum_{i,j\in N} E_{ij}+\ldots,\quad (a=1,\ldots,n+r),$$
(as long as the coefficient in front of the exceptional divisor is positive).
\end{defn}

\begin{lemma}\label{axfbafbf}\label{pushG}
For every $E\in D^b_{cusp}(X_N^{r})$, ${Rf_i}_*(E)\in D^b_{cusp}(X_{N\setminus i}^{r+1})$.
In particular, the collection $\hat\bG_N^{r}$ is contained in $D^b_{cusp}(X_N^r)$.  Moreover,  we have
$$R{f_i}_*{G_a^r}^\vee={G_a^{r+1}}^\vee\quad (a=1,\ldots,n+r).$$
In particular, $\hat\bG_N^{r}$ contains the line bundles ${G_a^r}^\vee$ ($1\leq a\leq n+r$)
and the following  torsion objects:
$$\cT=(i_{Z})_*\cL,\quad \cL=G_{a_1}^\vee\boxtimes\ldots\boxtimes G_{a_{k-1}}^\vee\boxtimes {G^r_{a_k}}^\vee$$
for all massive strata $Z=Z_{N_1,\ldots,N_k}$ in  $X^r_N$ and all $1\leq a_u\leq |N_u|-1$ when $u=1,\ldots,k-1$, and all 
$1\leq a_k\leq |N_k|+r$. 
\end{lemma}

\bp
The first statement follows from the commutative diagram
\begin{equation}\label{diagram}
\begin{CD}
X_N^r    @>f_i>>  X^{r+1}_{N\setminus i}\\
@V{\pi_j}VV        @VV{\pi_j}V\\
X_{N\setminus j}^r     @>f_i>>  X_{N\setminus i,j}^{r+1} 
\end{CD}
\end{equation}

To prove the rest of the Lemma, it suffices to prove 
 $R{f_i}_*{G_a^r}^\vee={G_a^{r+1}}^\vee$. Denote $t=\text{min}\{a-1, n\}$.  
Note that 
\begin{equation}
Lf_i^*{G_a^{r+1}}^\vee=f_i^*{G_a^{r+1}}^\vee=-aH+(a-1)\sum_{j\in N\setminus\{i\}} E_j+
\end{equation}
$$+(a-2)\sum_{j,k\in N\setminus\{i\}} E_{jk}+\ldots+
(a-t)\sum_{J\subseteq N\setminus\{i\}, |J|=t} E_{J},$$
as the pull-backs $f_i^*E_J$ are simply the proper transforms of the divisors $E_J$ under the blow-up map $f_i$. 
In particular, $f_i^*{G_a^{r+1}}^\vee={G_a^{r+1}}^\vee+F$, where 
$$F=(a-1)E_i+(a-2)\sum_{j\in N\setminus\{i\}} E_{ij}+\ldots+
(a-t)\sum_{J\subseteq N, i\in J, |J|=t} E_{J}.$$
Note that the coefficient $a-|J|$ of any $E_J$ appearing in $F$ satisfies 
$$1\leq a-|J|\leq n+r-j<n+r-j+1=\text{codim}_{X^r_N} E_J.$$
This implies, after repeatedly applying Lemma \ref{push exceptional} to the morphisms that successively blow down the divisors $E_J$ with $i\in J$, for a fixed $|J|$ (starting from the larger $|J|$ to the smaller), that $R{f_i}_*\cO(F)=\cO$. It follows that 
$$R{f_i}_*{G_a^r}^\vee=R{f_i}_*\big({f_i}^*{G_a^{r+1}}^\vee\otimes\cO(F)\big)={G_a^{r+1}}^\vee\otimes R{f_i}_*\cO(F)={G_a^{r+1}}^\vee.$$
\ep

The following Lemma is well known:
\begin{lemma}\label{push exceptional}
Let $p: X\ra Y$ be a blow-up of a smooth subvariety $Z$ of codimension $r+1$ of a smooth projective variety $Y$. Let $E$ be the exceptional divisor. Then 
for all $1\leq i\leq r$ we have
$$Rp_*\cO_X(iE)=\cO_Y.$$
\end{lemma}


\begin{lemma}\label{exceptional2}
All sheaves in $\hat\bG^r_N$ are exceptional.
\end{lemma}

The proof is identical to that of Lemma \ref{exceptional} and we omit it. 
The same direct computation as in Lemma \ref{G} shows that the line bundles $G^r_a$ satisfy the same restriction properties as the line bundles
$G_a$:
\begin{lemma}\label{restriction same}
For $S\subseteq N$, identifying the exceptional divisor $E_S\subseteq X^r_N$ with the product $\LM_S\times X^r_{N\setminus S}$, we have
$${G^r_a}_{|E_S}=\begin{cases}
G_{a}\boxtimes \cO & \text{ if }\quad a<|S|\cr
\cO & \text{ if }\quad a=|S|\cr
\cO\boxtimes G^r_{a-|S|}& \text{ if }\quad a>|S|.\end{cases}$$
\end{lemma}

The analogue of Lemma \ref{-G} is the following. 

\begin{lemma}\label{-G2} $\ $
\bi
\item[(1) ] Every $G^r_a$ is nef (hence, globally generated), of relative degree $1$ with respect to any forgetful map $\pi_i$, $i\in N$.
\item[(2) ] Each $(G^r_a)^\vee$ is acyclic. 
\item[(3) ] $R\Hom(G_a,G_b)=0$ if $a>b$.
\item[(4) ] $R\Gamma(-\psi_0+G_a-G_b)=0$ if $a<b$.
\ei
\end{lemma}
The proof is identical to that of Lemma \ref{-G} and we omit it. Note that unlike the case $r=-1$,  when $r\geq0$, there is no more $S_2$-symmetry, and it is generally false that $R\Hom(G_a,G_b)=0$ if $a<b$. For example, if $r\geq0$, then $R\Hom(G_1,G_2)=R\Ga(H-\sum_{i\in N} E_i)\neq0$. As a result, the order $<$ of Lemma \ref{order} does not generally descend to an order on $\hat\bG^r_N$ which makes $\hat\bG^r_N$ an exceptional collection (for example, if $N=\emptyset$). However, the order $<'$ from Corollary \ref{order2} descends to an order on $\hat\bG^r_N$ which makes $\hat\bG^r_N$ an exceptional collection:

\begin{lemma}
For  all $r\geq-1$, the set $\hat\bG^r_N$ is an exceptional collection with respect to the following order. 
Let $\cT, \cT'\in \hat\bG^r_N$. Let $(k_1,\ldots,k_t;a_1,\ldots,a_t)$ and $(k'_1,\ldots,k'_s;a'_1,\ldots,a'_s)$ be the corresponding data.
Then $\cT>'\cT'$ if the sequence $$(k_t-a_t,-k_t,k_{t-1}-a_{t-1},-k_{t-1},,\ldots)$$ is lexicographically(=alphabetically) larger than 
$$(k'_s-a_s',-k'_s,k'_{s-1}-a'_{s-1},-k'_{s-1},\ldots).$$
\end{lemma}

\bp
The proof is similar to that of Lemma \ref{order}. We sketch the proof for completeness. 
Let $Z$ and $Z'$ be massive strata supporting sheaves $\cT>\cT'$ in $\hat\bG$. 
These sheaves have the form
${Ri_Z}_*\cL$ and 
${Ri_{Z'}}_*\cL'$, respectively.
We~have to show that $R\Hom(\cT,\cT')=0$. 
Let $U$ be the smallest stratum containing both $Z$ and $Z'$.
Then $U$ is the intersection of codimension one strata (exceptional divisors) $D_1,\ldots,D_s$ containing both $Z$ and $Z'$.
Let $W=Z\cap Z'$. We can assume that $W$ is non-empty as otherwise there is nothing to prove.
Let $i_{W,Z}:\,W\hookrightarrow Z$ and $i_{W,Z'}:\,W\hookrightarrow Z'$
be the inclusions. As in the proof of Lemma \ref{order}, it suffices to prove that 
$$R\Hom(Li_{W,Z}^*\cL,Li_{W,Z'}^*\cL'(D+c_1(\cN)))=0,$$
where~$c_1(\cN)$ is the first Chern class of the normal bundle $\cN:=\cN_{W,Z'}$, i.e., 
the sum of all the exceptional that cut out $Z$ but don't contain $Z'$.
We write 
$$W=\LM_{K_1}\times\LM_{K_2}\times\ldots\LM_{K_{s-1}}\times X^r_{K_t},$$
$$R\Hom({Li_{W,Z}}^*\cL, {Li_{W,Z'}}^*\cL'(D+c_1(\cN)))=C_1\otimes C_2\otimes \ldots\otimes C_t,$$
where $C_1$ is $R\Hom$ between components of line bundles 
${Li_{W,Z}}^*\cL$ and\break ${Li_{W,Z'}}^*\cL'(D+c_1(\cN))$ corresponding to $\LM_{K_1}$, etc.

{\bf Case 1.} Suppose  $k_t-a_t>k'_t-a'_t$.
We prove that $C_t=0$. If $k_t<k'_t$, then $a'_t>a_t+(k'_t-k_t)>(k'_t-k_t)$. Hence, 
$$C_t=R\Hom(-G_{a_t}, -G_{a'_t-(k'_t-k_t)}-\psi_\inf),$$ 
where $-\psi_0$ is a contribution from $c_1(\cN)$ (there is no contribution from $D$). As $a'_t-(k'_t-k_t)>a_t$, it follows that $C_t=0$ by Lemma~\ref{-G2}~(4).

If $k_t=k_t'$, then $a'_t>a_t$. As there is no contribution from $c_1(\cN)$ to $C_t$, we have that 
either $C_t=R\Hom(-G_{a_t}, -G_{a'_t})=0$ (if $D$ doesn't include~$D_{K_1}$)
or $C_t=R\Hom(-G_{a_t}, -G_{a'_t}-\psi_0)=0$
(if $D$ includes $D_{K_1}$).

If $k_t>k'_t$, then there are no contributions from $c_1(\cN)$ or $D$ to $C_t$ and 
 $C_t=R\Hom(L, -G_{a'_t})=0$, where $L=-G_{a_t-(k_t-k'_t)}$ if $a_t>k_t-k'_t$ or $L=\cO$ otherwise. As by assumption 
 $a_t-(k_t-k'_t)<a'_t$, it follows that in both cases $C_t=0$ by Lemma~\ref{-G}.

{\bf Case 2.} Suppose  $k_t-a_t=k'_t-a'_t$ and $k_t<k'_t$. Then 
$a'_t=(k'_t-k_t)+a_t$ and hence, $a'_t>(k'_t-k_t)$.
As in Case 1, we have that 
$$C_t=R\Hom(-G_{a_t}, -G_{a'_t-(k'_t-k_t)}-\psi_0)=R\Hom(-G_{a_t}, -G_{a_t}-\psi_0)=0.$$

{\bf Case 3.} Suppose  $a_t=a'_t$, $k_t=k'_t$ and $D$ includes $D_{K_t}$.
In this case, we have $C_t=R\Hom(-G_{a_t}, -G_{a_t}-\psi_0)=0$.

{\bf Case 4.} Suppose  $a_t=a_t'$, $k_t=k_t'$, $D$ does not include $D_{K_t}$.
In this case $C_t=R\Hom(-G_{a_t}, -G_{a_t})=\CC$ is useless.
However, we can now use Corollary \ref{order2} as the remaining factors are Losev-Manin spaces (or alternatively, 
proceed exactly as above, by restricting to the next Losev--Manin factor $\LM_{K_{t-1}}$ in $W$). 
The Lemma follows.
\ep

\begin{lemma}
Let $r\ge -1$. For every $\cT\in \hat\bG_N^{r+1}$ and every $j\in N$, we have
$${R\pi_j}_*{Lf_i}^*\cT=0.$$
\end{lemma}

\bp
We use the commutative diagram (\ref{diagram}).
Since $\pi_j$ is flat and $X_{N\setminus j}^r\times_{X_{N\setminus ij}^{r+1}}X^{r+1}_{N\setminus i}$
has toroidal, and hence rational, singularities, the claim follows by 
cohomology and base change.
\ep

To finish the proof of Theorem~\ref{hhbhjghjgh}, we prove the following crucial result:
\begin{lemma}\label{crucial}
Let $r\ge -1$. For every $\cT\in \hat\bG_N^{r+1}$, 
$$\Cone\left [
{L\pi_i}^*{R\pi_i}_*{Lf_i}^*\cT\to {Lf_i}^*\cT
\right]$$
belongs to the subcategory
generated by $\hat \bG_N^{r}$.
\end{lemma}

We postpone the proof of Lemma \ref{crucial} to the end of this section. We use Lemma \ref{crucial} 
to prove the following result (that implies Theorem~\ref{hhbhjghjgh}):
\begin{prop}\label{full exc}
If $N\neq\emptyset$, the subcategory $D^b_{cusp}(X^r_N)$ is generated by $\hat\bG_N^{r}$.
\end{prop}

This proves the following theorem:
\begin{thm}
For all $r\geq-1$, $\hat\bG_N^{r}$ is a full, exceptional collection in $D^b_{cusp}(X^r_N)$. 
\end{thm}

In particular, when $r=-1$ this gives that the collection $\hat\bG_N$ is a full, exceptional collection in $D^b_{cusp}(\LM_N)$ 
(Theorem \ref{mainLM}). 

\bp[Proof of Prop. \ref{full exc}]
We argue by induction on the dimension $n+r$, and for a fixed $n+r$,
by induction on $n$.  The base of induction is $X_1^{r-1}$. Note that we have a $\bP^1$-bundle
$\pi_1:\,X_1^{r-1}\to\bP^{r-1}$. By Orlov's Theorem \ref{OrlovI}, $D^b_{cusp}(X_1^{r-1})$
is generated by 
$$\pi_1^*D^b(\bP^{r-1})\mathop{\otimes}\cO(-E_1)=
\langle
\cO(-rH+(r-1)E_1), \ldots, \cO(-2H+E_1), \cO(-H)\rangle,$$
which is precisely our claim in this case. 

Assume $n\geq2$. Choose an object $E\in D^b_{cusp}(X^r_N)$ such that $R\Hom(\cT,E)=0$ for every $\cT\in \hat\bG_N^{r}$.
We need to show that $E=0$.  We first show that ${Rf_i}_*E=0$ for all $i\in N$. Let $i\in N$. 
By Lemma~\ref{axfbafbf}, ${Rf_i}_*E\in D^b_{cusp}(X^{r+1}_{N\setminus\{i\}})$. By the inductive assumption, to prove 
${Rf_i}_*E=0$, it is sufficient to prove that $R\Hom(\cT, {Rf_i}_*E)=0$ for every $\cT\in \hat\bG_{N\setminus\{i\}}^{r+1}$. Note that 
$$R\Hom(\cT, {Rf_i}_*E)=R\Hom(Lf_i^*\cT, E).$$
If we let $C=\Cone\left [{L\pi_i}^*{R\pi_i}_*{Lf_i}^*\cT\to {Lf_i}^*\cT\right]$, it follows by Lemma~\ref{crucial} that
$R\Hom(C, E)=0$. Using the distinguished triangle 
$${L\pi_i}^*{R\pi_i}_*{Lf_i}^*\cT\ra {Lf_i}^*\cT\ra C\ra {L\pi_i}^*{R\pi_i}_*{Lf_i}^*\cT[1]$$
and the fact that  for all $F\in D^b(X^{r}_{N\setminus\{i\}})$ we have (since $E\in D^b_{cusp}(X^r_N)$)
$$R\Hom(L{\pi_i}^*F,E)=R\Hom(F, R{\pi_i}_*E)=0,$$
it follows that $R\Hom(Lf_i^*\cT, E)=0$. This proves that ${Rf_i}_*E=0$ for all $i\in N$. 
In particular, by Lemma \ref{support} the support $\Supp\, E$ of $E$ is contracted by 
all birational maps $f_i$, $i\in N$:

$$\Supp\, E\subseteq \Exc(f_1)\cap\ldots\cap \Exc(f_n).$$
Since $\Exc(f_i)=\bigcup_{i\in S}E_S$ and 
$E_S\cap E_T\ne\emptyset$ if and only if $S\subseteq T$ or $T\subseteq S$, 
this implies that $\Exc(f_1)\cap\ldots\cap \Exc(f_n)$ can be non-empty only if  
$r\ge1$, in which case this intersection is contained in $E_N$, the exceptional divisor 
corresponding to blowing up the proper transform
$\Delta_N\cong \LM_N$ of the subspace spanned by the points in $N$
(the last blow-up). It follows that 
$$\Supp\,E\subseteq E_N\cong \Delta_N\times \bP^r.$$

For every $i\in N$, we can decompose $f_i=f_i'\circ p$, where $p: X^r_N\ra Y$ blows-down $E_N$ (with image $\De_N$)
and $f_i': Y\ra X^{r+1}_{N\setminus\{i\}}$ is the composition of blow-downs of $E_S$ for $i\in S$, $S\ne N$. If we denote $E'_S=p(E_S)$,
it is still the case that $E'_S\cap E'_T\ne\emptyset$ if and only if $S\subseteq T$ or $T\subseteq S$. It follows that 
\begin{equation}\label{empty}
\Exc(f'_1)\cap\ldots\cap \Exc(f'_n)=\emptyset.
\end{equation}

Since ${Rf'_i}_*Rp_*E={Rf_i}_*E=0$ for all $i$, (\ref{empty}) and Lemma \ref{support} impliy that
$$Rp_*E=0.$$

Let $\alpha:\,E_N\hookrightarrow X^r_N$ be the inclusion map. By Orlov's Theorem~\ref{OrlovII}, 
$E$~belongs to the subcategory in $D^b(X^r_N)$ with semi-orthogonal decomposition
$$\left\langle 
R\alpha_*\left[D^b(\LM_N)\boxtimes\cO_{\bP^r}(-r)\right],\ldots,
R\alpha_*\left[D^b(\LM_N)\boxtimes\cO_{\bP^r}(-1)\right]\right\rangle.$$
In particular, there exist morphisms
$$0=E_0\to E_1\to\ldots\to E_r=E$$
that fit into exact triangles 
$$E_{i-1}\to E_i\to F_i \ra E_{i-1}[1]\quad\hbox{\rm with}\ F_i\in R\alpha_*\left[D^b(\LM_N)\boxtimes\cO_{\bP^r}(-i)\right].$$

\begin{claim}\label{cusp}
$F_i\in R\alpha_*\left[D^b_{cusp}(\LM_N)\boxtimes\cO_{\bP^r}(-i)\right]$ for all $1\leq i\leq r$.
\end{claim}

The proposition now follows immediately from Claim \ref{cusp}: by the inductive hypothesis, the subcategory 
$R\alpha_*\left[D^b_{cusp}(\LM_N)\boxtimes\cO_{\bP^r}(-i)\right]$ is generated by sheaves that belong to $\hat\bG_N^r$, but the latter have no non-zero morphisms into $E$. Thus  $E=0$.
\ep

\bp[Proof of Claim \ref{cusp}]
Let $F_i=R\alpha_*(H_i\boxtimes\cO_{\bP^r}(-i))$ for some $H_i\in D^b(\LM_N)$. We have to show that $H_i\in D^b_{cusp}(\LM_N)$ for all $i$. 

Let $j\in N$ and let $\alpha_j:\,\LM_{N\setminus\{j\}}\times \bP^r\hookrightarrow X^r_{N\setminus\{j\}}$ be the inclusion. Then 
$$R{\pi_j}_*F_i=R{\pi_j}_*R\alpha_*(H_i\boxtimes\cO_{\bP^r}(-i))=R{\alpha_j}_*(R{\pi_j}_*H_i\boxtimes\cO_{\bP^r}(-i)),$$
where $R{\pi_j}_*H_i\in D^b(\LM_{N\setminus\{j\}})$. In particular, $R{\pi_j}_*H_i=0$ if and only if $R{\pi_j}_*F_i=0$. 
Note that  ${R\pi_j}_*F_i$ belongs to the subcategory 
$$R{\alpha_{j}}_*\left[D^b(\LM_{N\setminus\{j\}})\boxtimes\cO_{\bP^r}(-i)\right].$$ 

Suppose  ${R\pi_j}_*F_p\ne0$ for some $j\in N$ and choose the maximal $p$ with this property.
Applying ${R\pi_j}_*$ to the filtration gives morphisms
$$0={R\pi_j}_*E_0\to {R\pi_j}_*E_1\to\ldots\to {R\pi_j}_*E_r={R\pi_j}_*E=0$$
that fit into exact triangles 
$${R\pi_j}_*E_{i-1}\to {R\pi_j}_*E_i\to {R\pi_j}_*F_i \ra {R\pi_j}_*E_{i-1}[1].$$
In particular, 
${R\pi_j}_*E_{p}={R\pi_j}_*E_{p+1}=\ldots={R\pi_j}_*E_{r}={R\pi_j}_*E=0$ and
$${R\pi_j}_*F_p\simeq {R\pi_j}_*E_{p-1}[1].$$
However, ${R\pi_j}_*E_{p-1}[1]$ belongs to the subcategory generated by 
$$R\alpha_{j}*\left[D^b(\LM_{N'})\boxtimes\cO_{\bP^r}(-i)\right]$$
for $i<p$, and thus can not have a non-zero morphism to ${R\pi_j}_*F_p$.
\ep

We now prove Lemma \ref{crucial}. The proof occupies the rest of this section. 
We first prove the case when $\cT'={G^{r+1}_a}^\vee$ on $X^{r+1}_{N\setminus\{i\}}$ in Lemma \ref{basic}. 

\begin{lemma}\label{basic}
Let $r\ge -1$. For all $1\leq a\leq n+r$, $i\in N$, we have
$${\pi_i}_*\big(f_i^*{G_a^{r+1}}^\vee\big)=0,$$
$${R^1\pi_i}_*\big(f_i^*{G_a^{r+1}}^\vee\big)={G_{a-1}^{r}}^\vee\oplus{G_{a-2}^{r}}^\vee\oplus\ldots\oplus{G_{1}^{r}}^\vee \text{ if } a\geq2 \text{ and }0 \text{ if }a=1,$$
$$\Cone\left [
{L\pi_i}^*{R\pi_i}_*{Lf_i}^*{G_a^{r+1}}^\vee\to {Lf_i}^*{G_a^{r+1}}^\vee\right]={G_{a}^{r}}^\vee\oplus{G_{a-1}^{r}}^\vee\oplus\ldots\oplus{G_{1}^{r}}^\vee.$$
\end{lemma}

\bp
If $a=1$, then $f_i^*{G_1^{r+1}}^\vee=-H={G_1^{r}}^\vee$ and the statements follow at once as 
${G_1^{r}}^\vee$ is a cuspidal object.  
Assume now $a\geq2$. For clarity, consider first the situation when $a\leq n$. In this case, we have
$$(G_a^r)^\vee=-aH+(a-1)\sum_{j\in N\setminus i} E_j+\ldots+1\cdot\sum_{J\subseteq N\setminus i, |J|=a-1} E_J.$$
Define divisors on $X_N^r$ as follows: 
$$E^s=\sum_{i\in I\subseteq N, |I|=s}E_I\quad (1\leq s\leq a-1),$$
$$F_s=E^1+E^2+\ldots+E^{s}=E_i+\sum_j E_{ij}+\sum_{j,k} E_{ijk}+\ldots+\sum_{i\in I\subseteq N, |I|=s}E_I.$$

$$H_1=f_i^*{G_a^{r+1}},\quad H_2=H_1+F_1,\quad H_3=H_2+F_2,\quad\ldots, H_a=H_{a-1}+F_{a-1}={G_a^{r}}^\vee.$$

There are two sets of exact sequences that we will use, identifying as usual, divisors with the corresponding line bundles:
\begin{itemize}
\item[(A) ]
$$0\ra H_1\ra H_2\ra {H_2}_{| F_1}\ra 0,$$
$$0\ra H_2\ra H_3\ra {H_3}_{| F_2}\ra 0,$$
$$\vdots$$
$$0\ra  H_{a-1}\ra  H_a\ra {H_a}_{| F_{a-1}}\ra 0.$$

\item[(B) ] For $2\leq k\leq a-1$ and $1\leq s\leq k-1$, letting
$$Q^k_{s+1}= \big(H_{k+1}-F_s\big)_{|E^{s+1}}=\bigoplus_{i\in I\subseteq N, |I|=s+1} \big(H_{k+1}-F_s\big)_{|E_I},$$
we have exact sequences
$$0\ra Q^k_2\ra {H_{k+1}}_{|F_2}\ra {H_{k+1}}_{|F_1}\ra0,$$ 
$$0\ra Q^k_3\ra {H_{k+1}}_{|F_3}\ra {H_{k+1}}_{|F_2}\ra0,$$ 
$$\vdots$$
$$0\ra Q^k_{k}\ra {H_{k+1}}_{|F_k}\ra {H_{k+1}}_{|F_{k-1}}\ra0,$$ 
\end{itemize}
Note that $F_{s+1}=F_s+E^{s+1}$ and the sequences in (B) are obtained by tensoring with $H_{k+1}$ the canonical sequence 
\begin{equation}\label{canonical}
0\ra \cO(-F_s)_{|E^{s+1}}\ra \cO_{F_{s+1}}\ra \cO_{F_s}\ra0.
\end{equation}

Lemma \ref{basic} follows at once from taking $k=1$ in (A1), (A2) in Claim \ref{claimAB}. Here parts (B1)-(B3) refer to the exact sequences in (B), while (A1)-(A2) refer to the exact sequences in (A). Parts (B1)-(B3) will be used to prove (A1)-(A2) (this is why they appear first).

We now discuss the case when $a>n$. In this case, we have 
$$(G_a^r)^\vee=-aH+(a-1)\sum_{j\in N\setminus i} E_j+\ldots+(a-(n-1))\cdot E_{N\setminus i}.$$

We define $F_s$ as above in the range $1\leq s\leq n$, and let 
$$F_{a-1}=\ldots=F_{n+1}=F_n=E^1+E^2+\ldots+E^n$$

We define $H_k$ as above, for all $1\leq k\leq a$. As before, $H_a=(G_a^r)^\vee$. 
We use the exact sequences as in (A). In order to analyze the sheaves ${H_{k+1}}_{|F_k}$, there are 
two cases to consider: (1) $1\leq k\leq n<a$, and (2) $n<k\leq a-1$. For a fixed $k$, we consider the sequences (\ref{canonical})
as in (B), where for $1\leq s\leq k-1$, $Q^k_{s+1}$ is defined as before if $s\leq n-1$, while 
if $k-1\geq s\geq n$, we let
$$Q^k_{n+1}=\ldots=Q^k_k=0.$$
Hence, the exact sequences in (B) that we consider are:
$$0\ra Q^k_2\ra {H_{k+1}}_{|F_2}\ra {H_{k+1}}_{|F_1}\ra0,$$ 
$$0\ra Q^k_3\ra {H_{k+1}}_{|F_3}\ra {H_{k+1}}_{|F_2}\ra0,$$ 
$$\vdots$$
$$0\ra Q^k_{n}\ra {H_{k+1}}_{|F_n}\ra {H_{k+1}}_{|F_{n-1}}\ra0,$$ 
The rest of the proof is identical, as Claim \ref{claimAB} still holds. 
\ep

\begin{claim}\label{claimAB}
\begin{itemize}
\item[(B1) ] For $2\leq k\leq a-1$ and $1\leq s\leq \text{min}\{k-1,n-1\}$, we have 
$$R{\pi_i}_*Q^k_{s+1}=0.$$ 

\item[(B2) ] For all $1\leq  s\leq k\leq a-1$, we have
$${R^l\pi_i}_*\big({H_{k+1}}_{|F_s}\big)=0\quad\text{ for all } l>0,$$
$${\pi_i}_*\big({H_{k+1}}_{|F_s}\big)=(G_{k}^{r})^\vee.$$

\item[(B3) ] For all  $1\leq  s\leq k\leq a-1$, the canonical map 
$${\pi_i}^*{\pi_i}_*\big({H_{k+1}}_{|F_s}\big)\ra \big({H_{k+1}}_{|F_s}\big)$$ is surjective
with kernel $\pi_i^*(G_k^r)^\vee\otimes\cO(-F_s)$. Moreover, 
$$\Cone\left [
{L\pi_i}^*{R\pi_i}_*\big({H_{k+1}}_{|F_s}\big)\ra \big({H_{k+1}}_{|F_s}\big)\right]=\big(\pi_i^*(G_k^r)^\vee\otimes\cO(-F_s)\big)[1].$$
In particular, 
$$\Cone\left [
{L\pi_i}^*{R\pi_i}_*\big({H_{k+1}}_{|F_k}\big)\ra \big({H_{k+1}}_{|F_k}\big)\right]=(G_{k}^{r})^\vee[1].$$

\item[(A1) ] For all $1\leq k\leq a-1$, we have
$${\pi_i}_*\big(H_k\big)=0,$$ 
$${R^1\pi_i}_*\big(H_k\big)=(G_{a-1}^{r})^\vee\oplus(G_{a-2}^{r})^\vee\oplus\ldots\oplus(G_{k}^{r})^\vee.$$

\item[(A2) ] For all $1\leq k\leq a$, we have
$$\Cone\left [
{L\pi_i}^*{R\pi_i}_*\big(H_k\big)\ra \big(H_k\big)\right]=(G_a^{r})^\vee\oplus(G_{a-1}^{r})^\vee\oplus\ldots\oplus(G_{k}^{r})^\vee.$$
\end{itemize}
\end{claim}

\bp[Proof of Claim \ref{claimAB}] 
We prove (B1)-(B3). From the commutative diagram
\begin{equation}\label{push1}
\begin{CD}
\LM_I\times X^r_{N\setminus I}=E_I    @>>>  X^r_N\\
@V{(\pi_i,Id)}VV        @VV{\pi_i}V\\
\LM_J\times X^r_{N\setminus I}=E_J     @>>>  X_{N\setminus i}^{r} 
\end{CD}
\end{equation}
it follows that 
$$R{\pi_i}_*\big(-\psi_x\boxtimes(G^r_{k-s})^\vee\big)=R{\pi_i}_*(-\psi_x)\boxtimes(G^r_{k-s})^\vee=0,$$
as $R{\pi_i}_*(-\psi_x)=0$. 
Hence, (\ref{comp3}) implies that $R{\pi_i}_*Q^k_{s+1}=0$, thus proving (B1). 
Note that it suffices to prove (B2) and (B3) for 
$1\leq s\leq \text{min}\{k,n\}$, as $F_n=F_{n+1}=\ldots=F_{a-1}$. 
Clearly, (B2) follows immediately from (B1), the exact sequences in (B) and the diagram (\ref{comp4}). 
We now prove (B3)  by induction on $s$ (for a fixed $k$). Denote
$$h_s: {\pi_i}^*{\pi_i}_*\big({H_{k+1}}_{|F_s}\big)\ra \big({H_{k+1}}_{|F_s}\big),\quad \cK_s=\Ker(h_s)$$ 

We use the following two observations: (1) 
For any sheaf $\cT$, the canonical map ${\pi_i}^*{\pi_i}_*(\cT)\ra \cT$ is non-zero whenever ${\pi_i}_*(\cT)$ is non-zero, and (2) If $F\subset X$ is an effective divisor and $\cL$ is a line bundle on $X$,  the only non-zero morphism $\cL\ra\cL_{|F}$ is 
the restriction map (with kernel $\cL(-F)$).

When $s=1$, we have from (B2) and  (\ref{comp4})  that 
$${\pi_i}_*\big({H_{k+1}}_{|F_1}\big)=(G^r_k)^\vee,\quad {\pi_i}^*{\pi_i}_*\big({H_{k+1}}_{|F_1}\big)= {\pi_i}^*(G^r_k)^\vee,$$
$${H_{k+1}}_{|F_1}=(G^r_k)^\vee=({{\pi_i}^*(G^r_k)}^\vee)_{|F_1}\quad\text{ on }F_1=E_i.$$
Hence, it follows from the observations (1) and (2) above that  $h_1$ is surjective and 
$\cK_1=\pi_i^*(G_k^r)^\vee\otimes\cO(-F_1)$. 
Assume now that $h_s$ is surjective and $\cK_s=\pi_i^*(G_k^r)^\vee\otimes\cO(-F_s)$. 
By applying $\pi_i^*{\pi_i}_*(-)$ to the exact sequence 
\begin{equation}\label{Q}
0\ra Q^k_{s+1}\ra \big(H_{k+1}\big)_{|F_{s+1}}\ra \big(H_{k+1}\big)_{|F_s}\ra0,
\end{equation}
it follows from (B1) that there is a commutative diagram:
\begin{equation*}
\begin{CD}
0 @>>> 0   @>>>  {\pi_i}^*{\pi_i}_*\big({H_{k+1}}_{|F_{s+1}}\big)@>>>  {\pi_i}^*{\pi_i}_*\big({H_{k+1}}_{|F_{s}}\big)@>>> 0\\
  @.                 @VVV      @VV{h_{s+1}}V   @VV{h_{s}}V\\
0 @>>>  Q^k_{s+1}    @>>>  \big({H_{k+1}}_{|F_{s+1}}\big)@>>> \big({H_{k+1}}_{|F_{s}}\big)@>>> 0
\end{CD}
\end{equation*}
By our inductive assumption, $h_s$ is surjective. By the snake lemma, there is an exact sequence
$$0\ra \cK_{s+1}\ra\cK_s\ra Q^k_{s+1}\ra \Coker(h_{s+1})\ra0.$$
The induced map $\cK_s\ra Q^k_{s+1}$ is non-zero. Otherwise, $Q^k_{s+1}\cong\Coker(h_{s+1})$, which implies that
the exact sequence  (\ref{Q}) is split, since there is a retract $ \big({H_{k+1}}_{|F_{s+1}}\big)\ra Q^k_{s+1}$. But the sequence 
(\ref{Q}) is obtained by tensoring the canonical sequence (\ref{canonical}) with a line bundle, and (\ref{canonical}) is not split, as there are no 
non-zero morphisms $\cO_{F_{s+1}}\ra\cO_{E^{s+1}}(-F_s)$:
$$\Hom\big(\cO_{F_{s+1}}, \cO_{E^{s+1}}(-F_s)\big)=\HH^0\big(\cO_{E^{s+1}}(-F_s)\big)=0,$$
by (\ref{comp2}), and we have a contradiction. 
By the induction assumption, we have $\cK_s=\pi_i^*(G_k^r)^\vee\otimes\cO(-F_s)$. By (\ref{comp3}), 
we have that 
$$Q^k_{s+1}=\big(\pi_i^*(G_k^r)^\vee\otimes\cO(-F_s)\big)_{|E^{s+1}}.$$ 
Hence, $Q^k_{s+1}=(\cK_s)_{|E^{s+1}}$. By observation (2), the map $\cK_s\ra Q^k_{s+1}$ is surjective, i.e., 
 $\Coker(h_{s+1})=0$ and furthermore, 
$$\cK_{s+1}=\cK_s(-E^{s+1})=\pi_i^*(G_k^r)^\vee\otimes\cO(-F_s-E^{s+1})=\pi_i^*(G_k^r)^\vee\otimes\cO(-F_{s+1}),$$
This proves the first statement in (B3). 
In particular, this gives
$$\Cone\left [
{L\pi_i}^*{R\pi_i}_*\big({H_{k+1}}_{|F_k}\big)\ra \big({H_{k+1}}_{|F_k}\big)\right]=\big(\pi_i^*(G_k^r)^\vee\otimes\cO(-F_k)\big)[1].$$
and now the last statement in (B3) follows from
$${G_{k}^{r}}^\vee=\pi_i^*(G_k^r)^\vee\otimes\cO(-F_k).$$  

We now prove (A1) and (A2).  Apply ${\pi_i}_*(-)$ to the exact sequences in (A). 
Then (A1) follows from (B2) and downward induction, using the fact that there are no non-trivial extensions 
between $(G^r_k)^\vee$ and $(G^r_{k'})^\vee$ for $k\neq k'$.  Similarly, to prove (A2), we use downward induction on 
$1\leq k\leq a$ and the exact sequences in (A). As $H_a=(G^r_a)^\vee$, we have
$$\Cone\left [{L\pi_i}^*{R\pi_i}_*\big(H_a\big)\ra \big(H_a\big)\right]=(G_a^{r})^\vee.$$
Note that if  $\pi:X\ra Y$ is a  morphism between smooth projective varieties and $0\ra A_1\ra A_2\ra A_3\ra 0$ is an exact sequence of sheaves on $X$,
there is a distinguished triangle relating the cones $C_i=\Cone\left [{L\pi}^*{R\pi}_*{A_i}\ra {A_i}\right]$:
$$C_1\ra C_2\ra C_3\ra C_1[1]$$
Then (A2) follows from (B3) by using the fact that there are no non-trivial extensions between ${G^r_k}^\vee$ and ${G^r_{k'}}^\vee$ for $k\neq k'$.
\ep

\begin{claim}\label{comps}
For all subsets $I\subseteq N$ with $i\in I$, $|I|=s+1$, $1\leq s\leq n-1$, we have 
on $E_I\cong\LM_I\times X^r_{N\setminus I}$
\begin{equation}\label{comp2}
{F_s}_{|E_I}=\psi_x\boxtimes\cO,
\end{equation}
(here $x$ is the attaching point). 
Assume now that $1\leq k\leq a-1$ and $1\leq s\leq \text{min}\{k-1,n-1\}$. Then
\begin{equation}\label{comp1}
{H_{k+1}}_{|E_I}=\cO\boxtimes (G^r_{k-s})^\vee=\big(\pi_i^*(G^r_k)^\vee\big)_{|E_I},
\end{equation}
Hence, we have 
$$\big(H_{k+1}-F_s\big)_{|E_I}=(-\psi_x)\boxtimes (G^r_{k-s})^\vee=\big(\pi_i^*(G^r_k)^\vee\otimes\cO(-F_s)\big)_{|E_I},$$
\begin{equation}\label{comp3}
Q^k_{s+1}=\bigoplus_{i\in I\subseteq N, |I|=s+1}(-\psi_x)\boxtimes (G^r_{k-s})^\vee=\big(\pi_i^*(G^r_k)^\vee\otimes\cO(-F_s)\big)_{|E^{s+1}}.
\end{equation}
Moreover, on $E_i\cong X^r_{N\setminus\{i\}}$ we have 
\begin{equation}\label{comp4}
{H_{k+1}}_{|E_i}=(G^r_k)^\vee.
\end{equation}
\end{claim}

\bp
To prove  (\ref{comp2}), we let $I=J\cup\{i\}$. Then $|J|=s$. We have:
$$(F_s)_{|E_{I}}=\big(E_i+\sum_{j\in N\setminus i}E_{ij}+\ldots+\sum_{K\subseteq N\setminus i, |K|=s-1}E_K\big)_{|E_{I}}=$$
$$=\de_{J\cup\{x\}}+\sum_{j\in J}\de_{(J\setminus\{j\})\cup\{x\}}+\ldots+\sum_{j\in J}\de_{jx}.$$
(as divisors on $\LM_I$). 
Using the $\psi_x$ Kapranov model of $\LM_I$, one has
$$(F_s)_{|E_{I}}=\La_J+\sum_{j\in J}E_{J\setminus\{j\}}+\ldots+\sum_{j\in J}E_{j}=H.$$
Here $\La_J$  denotes the class of the proper transform in $\LM_I$ 
of the hyperplane in $\PP^s$ spanned by the points in $J$. This proves (\ref{comp2}). 

To see (\ref{comp1}) and (\ref{comp4}), recall that if $1\leq k\leq a-1$, then 
$$H_{k+1}=H_1+F_1+\ldots+F_k,$$
$$H_1=-aH+(a-1)\sum_{j\in N\setminus\{i\}} E_j+(a-2)\sum_{j,k\in N\setminus\{i\}} E_{jk}+\ldots+
(a-t)\sum_{K\subseteq N\setminus\{i\}, |K|=t} E_{K},$$
where $t=\text{min}\{a-1,n-1\}$. There are two cases to consider: 
\begin{itemize}
\item[(a) ] $k\leq n$ (with either $n\leq a-1$ or $a-1\leq n$),  
\item[(b) ] $k>n$, in which case, we must have $a-1>n$ and $t=n-1$. Note that we must have $r\geq2$, as $n+r-1\geq a-1>n$. 
\end{itemize}

In case (a), we have 
$$F_1+\ldots+F_k=kE_i+(k-1)\sum_{j\in N\setminus\{i\}} E_{ij}+\ldots+1\cdot \sum_{i\in K\subseteq N, |K|=k} E_K.$$

In case (b), we have
$$F_1+\ldots+F_k=kE_i+(k-1)\sum_{j\in N\setminus\{i\}} E_{ij}+\ldots$$
$$\ldots+(k-n+2)\sum_{i\in K\subseteq N, |K|=n-1} E_K+
(k-n+1)E_N.$$

Let now $1\leq s\leq \text{min}\{k-1,n-1\}$ and let $I\subseteq N$, with $i\in I$ and $|I|=s+1$, for some $0\leq s\leq k-1$. 
Let $$\cO(H_{k+1})_{|E_I}=\cO(H')\boxtimes \cO(H''),$$
where $H'$ is the component on $\LM_I$ and $H''$ is the component on $X^r_{N\setminus I}$.  
We now compute $H'$ and $H''$. Note that only the divisors $E_K$ with $I\subseteq K\subseteq N$ 
contribute to $H''$. For example, $H_1$ does not, i.e., we have:
$$H''=\big(F_1+\ldots+F_k\big)_{|E_I}.$$
In case (a) we have:
$$H''=(k-s)(-\psi_x)+(k-s-1)\sum_{k\in N\setminus I} \de_{k,x}+$$
$$+(k-s-2)\sum_{k,l\in N\setminus I} \de_{k,l,x}+\ldots+1\cdot\sum_{K\subseteq N\setminus I, |K|=k-s-1} \de_{K\cup\{x\}}.$$
 (as a divisor on $\M_{0,N\setminus I}$), while in case (b), we have:
 $$H''=(k-s)(-\psi_x)+(k-s-1)\sum_{k\in N\setminus I} \de_{k,x}+$$
$$+(k-s-2)\sum_{k,l\in N\setminus I} \de_{k,l,x}+\ldots+(k+1-n)\de_{(N\setminus I)\cup\{x\}}.$$
In both cases $H''=(G^r_{k-s})^\vee$, as by the definition of $G^r_{k-s}$ on $X^r_{N\setminus I}$,
$$(G^r_{k-s})^\vee=(k-s)H-(k-s-1)\sum_{j\in N\setminus I} E_j-\ldots-(k-s-t')\sum_{K\subseteq N\setminus I, |K|=t'} E_K,$$
(in the Kapranov model given by $\psi_x$),
where $t'=\text{min}\{k-s-1,n-s-1\}$, i.e., $t'=k-s-1$ in case (a), and $t'=n-s-1$ in case (b) 

We now calculate $H'$. Let $I=J\cup\{i\}$. Since $|J|=s\leq\text{min}\{k-1,n-1\}$, we have 
$s \leq t=\text{min}\{a-1,n-1\}$. Using the $\psi_0$ Kapranov model of $\LM_I$, we obtain that the contribution from $H_1$ to $H'$
comes from
$$-aH+(a-1)\sum_{j\in J} E_j+\ldots+(a-(s-1))\sum_{K\subseteq J, |K|=s-1} E_{K}+(a-s)E_J,$$
and equals
$$-aH+(a-1)\sum_{j\in J} E_j+\ldots+(a-(s-1))\sum_{K\subseteq J, |K|=s-1} E_{K}+(a-s)\La_J,$$
while the contribution from $F_1+\ldots+F_k$ to $H'$ comes from
$$kE_i+(k-1)\sum_{j\in J} E_{ij}+\ldots+(k-(s-2))\sum_{K\subseteq J, |K|=s-2} E_{K\cup\{i\}}+$$
$$+(k-(s-1))\sum_{K\subseteq J, |K|=s-1} E_{K\cup\{i\}}+(k-s)E_J,$$
and equals
$$kE_i+(k-1)\sum_{j\in J} E_{ij}+\ldots+(k-(s-2))\sum_{K\subseteq J, |K|=s-2} E_{K\cup\{i\}}+$$
$$+(k-(s-1))\sum_{K\subseteq J, |K|=s-1} \La_{K\cup\{i\}}+(k-s)(-\psi_x).$$
Here $\La_S$ (for $S\subseteq I$ with $|S|=s$) denotes the class of the proper transform in $\LM_I$ 
of the hyperplane in $\PP^s$ spanned by the points in $S$, i.e., 
$$\La_S=H-\sum_{K\subseteq S, 1\leq |K|\leq s-1} E_K.$$ 
We now sum up these two terms and compute the coefficient of $H$ to be:
$$-a+(a-s)+(k-s+1)s-(k-s)s=0.$$
Here we use that on $\LM_I$, the class of $\psi_x$ in the $\psi_0$ Kapranov model is 
$$\psi_x=sH-(s-1)\sum_{j\in I}-(s-2)\sum_{j,k\in I}-\ldots.$$
Similarly, the coefficient of $E_K$, for $K\subseteq J$, $|K|=l$ is:
$$(a-l)-(a-s)-(k-s+1)(s-l)+(k-s)(s-l)=0,$$
while the coefficient of $E_{K\cup\{i\}}$, for $K\subseteq J$, $|K|=l$ is:
$$(k-l)-(k-s+1)(s-l)+(k-s)(s-l-1)=0.$$
Hence, $H'=0$ and $\cO(H_{k+1})_{|E_I}=\cO\boxtimes (G^r_{k-s})^\vee$. 

To see $\big(\pi_i^*(G^r_k)^\vee\big)_{|E_I}=\cO\boxtimes (G^r_{k-s})^\vee$, 
we use the commutative diagram (\ref{push1}). Note that the line bundle 
${(G^r_k)^\vee}_{|E_J}=\cO\boxtimes (G^r_{k-s})^\vee$ (Remark \ref{restriction same}).
This finishes the proof of  (\ref{comp1}). The case when $E_I=E_i$ corresponds to the case $s=0$, and the above computation shows (\ref{comp4}). 
Clearly, (\ref{comp3}) follows from (\ref{comp1}) and (\ref{comp2}). 
\ep

To prove the general case of Lemma \ref{crucial}, we need the following:
\begin{lemma}\label{cuspidal}
If $\pi_i: X_N^r\ra X_{N\setminus i}^r$ is the forgetful map, for all $1\leq a\leq n+r-1$, the line bundle $\pi_i^*({G^r_a}^\vee)\otimes\cO(-E_i)$ 
belongs to the subcategory generated by $\hat \bG_N^{r}$.
\end{lemma}


\bp
Let $t=\text{min}\{a-1,n-1\}$. Keeping the notations in the proof of Lemma \ref{basic}, 
consider the divisors $F_s$ on $X_N^r$, for  $1\leq s\leq t+1$:
$$F_s=E^1+E^2+\ldots+E^{s}=E_i+\sum_j E_{ij}+\sum_{j,k} E_{ijk}+\ldots+\sum_{i\in I\subseteq N, |I|=s}E_I,$$
and let 
$$L^1=\pi_i^*({G^r_a}^\vee)-F_1,\quad L^2=\pi_i^*({G^r_a}^\vee)-F_2,\quad\ldots\quad, L^{t+1}=\pi_i^*({G^r_a}^\vee)-F_{t+1}.$$

We claim that $\pi_i^*({G^r_a}^\vee)-F_{t+1}=(G_a^{r})^\vee$. This is clear if one considers separately the two cases, $a\leq n$ and $a>n$. 
For example, if $a>n$ then 
$$(G_a^{r})^\vee=-aH+(a-1)\sum_{j\in N}+\ldots+(a-n)E_N=\pi_i^*{(G^r_a)^\vee}-F_n.$$

We have to prove that $L^1$ belongs to the subcategory generated by $\hat \bG_N^{r}$. We use the exact sequences
$$0\ra L^2\ra L^1\ra\bigoplus_{j\in N\setminus i}(L^1)_{|E_{ij}}\ra0,$$
$$0\ra L^3\ra L^2\ra\bigoplus_{j,k\in N\setminus i}(L^2)_{|E_{ijk}}\ra0,$$
$$\vdots$$
$$0\ra L^{t+1}\ra L^t\ra\bigoplus_{J\subseteq N\setminus\{i\}, |J|=t}(L^t)_{|E_{J\cup\{i\}}}\ra0.$$
Clearly, it is enough to prove that  the sheaves $(L^s)_{|E_{J\cup\{i\}}}$, for all $1\leq s\leq t$ and 
$J\subseteq N\setminus\{i\}$, $|J|=s$, are in the subcategory generated by $\hat \bG_N^{r}$. 
Note that $E_{J\cup\{i\}}$ is a massive stratum in $X^r_N$, as 
$$|J\cup\{i\}|=s+1\geq2,\quad |N\setminus J|+r=n-s+r>0,$$ 
since $s\leq t\leq n-1$ and when $r=-1$, we have $t=a-1$ and $a\leq n-1$. 

As in (\ref{comp1}), we have
$\big(\pi_i^*{G^r_a}^\vee\big)_{|E_{J\cup\{i\}}}=\cO\boxtimes (G^r_{a-s})^\vee,$
while by (\ref{comp2}), we have 
$\cO(-F_s)_{|E_{J\cup\{i\}}}=(-\psi_x)\boxtimes\cO$. 
It follows that $(L^s)_{|E_{J\cup\{i\}}}$ is one of the objects in $\hat \bG_N^{r}$, as it
equals $(-\psi_x)\boxtimes (G^r_{a-s})^\vee$.
\ep

\bp[Proof of Lemma \ref{crucial}]
Consider the case when $\cT$ is a torsion sheaf. Let
$$\cT=(i_{Z})_*\cL,\quad \cL=G_{a_1}^\vee\boxtimes\ldots\boxtimes G_{a_{l-1}}^\vee\boxtimes {G^{r+1}_{a_l}}^\vee,$$
where $Z=Z_{N_1, N_2,\ldots, N_{l}}$ is the massive stratum in $X^{r+1}_{N\setminus\{i\}}$ corresponding to a partition $N_1\sqcup\ldots\sqcup N_l$ 
of $N\setminus\{i\}$. Since $Z$ is massive, we have $|N_t|\geq2$ for every $1\leq t\leq l-1$ and $|N_l|+r+1>0$. 
The preimage $Z'=f_i^{-1}(Z)$ is a massive stratum in $X^r_N$ and there is a commutative diagram:
\begin{equation*}
\begin{CD}
Z'=f_i^{-1}(Z') @>{i_{Z'}}>>  X_N^r  \\
@V{Id\times f^{N_l}_i}VV        @VV{f_i}V\\
Z  @>{i_{Z}}>>   X^{r+1}_{N\setminus i}
\end{CD}
\end{equation*}
where $i_{Z'}$ and $i_Z$ are the canonical inclusions and we identify 
$$Z=\LM_{N_1}\times\ldots\times\LM_{N_{l-1}}\times X^{r+1}_{N_l},$$
$$Z'=\LM_{N_1}\times\ldots\times\LM_{N_{l-1}}\times X^r_{N_l\cup\{i\}},$$
and $f_i^{N_l}$ denotes the blow-up map $X^r_{N_l\cup\{i\}}\ra X^{r+1}_{N_l}$ (we write $f_i$ whenever there is no risk of confusion). 
Let $\cT'=Lf_i^*\cT'$. Then 
$$\cT'=(i_{Z'})_*\cL',\quad \cL'=(Id\times f_i)^*\cL'=G_{a_1}^\vee\boxtimes\ldots\boxtimes G_{a_{l-1}}^\vee\boxtimes f_i^*{{G^{r+1}_{a_l}}^\vee}.$$

We compute $\Cone\left [{L\pi_i}^*{R\pi_i}_*\cT'\to \cT'\right]$ by the exact same argument as in the proof of Lemma \ref{basic}. 
We define divisors $H_1,H_2,\ldots, H_{a_l}$ on $X^r_{N_l\cup\{i\}}$ exactly as before,  such that we have
$$H_1= f_i^*({G^{r+1}_{a_l})^\vee},\quad  H_{a_l}=(G_{a_l}^{r})^\vee.$$

Considering on $X_{N_l\cup\{i\}}^r$ the exact sequences (A) in the proof of Lemma  \ref{basic}, after taking the box product with 
$G_{a_1}^\vee\boxtimes\ldots\boxtimes G_{a_{l-1}}^\vee$, one obtains exact sequences on $Z'$. It is enough to prove 
that for all $1\leq k\leq a_l-1$, $\Cone\left [{L\pi_i}^*{R\pi_i}_*\cT_k\to \cT_k\right]$ is in the subcategory generated by 
$\hat \bG_N^{r}$, where 
$$\cT_k=(i_{Z'})_*\big(G_{a_1}^\vee\boxtimes\ldots\boxtimes G_{a_{l-1}}^\vee\boxtimes ({H_{k+1}}_{| F_k)}\big).$$
We consider on $X_{N_l\cup\{i\}}^r$ the exact sequences (B) in the proof of Lemma  \ref{basic}, after taking the box product with 
$G_{a_1}^\vee\boxtimes\ldots\boxtimes G_{a_{l-1}}^\vee$. Let 
$$\cT_{k,s}=(i_{Z'})_*\big(G_{a_1}^\vee\boxtimes\ldots\boxtimes G_{a_{l-1}}^\vee\boxtimes Q^k_{s+1}\big),$$ 
$$\tilde{\cT}_k=(i_{Z'})_*\big(G_{a_1}^\vee\boxtimes\ldots\boxtimes G_{a_{l-1}}^\vee\boxtimes ({H_{k+1}}_{|F_1})\big).$$ 
Then $\Cone\left [{L\pi_i}^*{R\pi_i}_*\cT_k\to \cT_k\right]$ is in the subcategory generated by 
$\hat \bG_N^{r}$ if and only if $\Cone\left [{L\pi_i}^*{R\pi_i}_*\cT_{k,s}\to \cT_{k,s}\right]$ and  
$\Cone\left [{L\pi_i}^*{R\pi_i}_*\tilde{\cT}_k\to \tilde{\cT}_k\right]$ are in the subcategory generated by 
$\hat \bG_N^{r}$. By (\ref{comp3}), the sheaf $Q^k_{s+1}$ is a direct sum of objects in $\hat \bG^r_{N_l\cup\{i\}}$. 
Hence, $\cT_{k,s}$ is a direct sum of objects in 
$\hat \bG_N^{r}$. In particular, $\Cone\left [{L\pi_i}^*{R\pi_i}_*\cT_{k,s}\to \cT_{k,s}\right]=\cT_{k,s}$. 
We are left to prove that 
$$\Cone\left [{L\pi_i}^*{R\pi_i}_*\tilde{\cT}_k\to \tilde{\cT}_k\right]$$  is in the subcategory generated by 
$\hat \bG_N^{r}$. 

For simplicity, denote $\tilde{\cT}=\tilde{\cT}_k$. 
Let $\bZ:=\pi_i(Z')$. We identify
$$\bZ=\LM_{N_1}\times\ldots\times\LM_{N_{l-1}}\times X^r_{N_l}.$$ 
Then $\pi_i^{-1}(\bZ)=Z^1\cup\ldots\cup Z^l$, where 
$$Z^l=Z'=\LM_{N_1}\times\ldots\times\LM_{N_{l-1}}\times X^r_{N_l\cup\{i\}},$$ 
$$Z^t=\LM_{N_1}\times\ldots\times\LM_{N_t\cup\{i\}}\times\ldots\times\LM_{N_{l-1}}\times X^r_{N_l}, \quad  1\leq t\leq l-1.$$
As the divisor $E_i$ in $X_{N_l\cup\{i\}}^r$ can be identified with $\LM_{\{i\}}\times X^r_{N_l}$, 
the sheaf $\tilde{\cT}$ is supported on the non-massive stratum
$$Z^l\cap Z^{l-1}=\LM_{N_1}\times\ldots\times\LM_{N_{l-1}}\times\LM_{\{i\}}\times X^r_{N_l},$$
$$\tilde{\cT}=(i_{Z^l\cap Z^{l-1}})_*\cM,\quad \cM=G_{a_1}^\vee\boxtimes\ldots\boxtimes G_{a_{l-1}}^\vee\boxtimes\cO\boxtimes {G^r_k}^\vee,$$ 
where $i_{Z^l\cap Z^{l-1}}: Z^l\cap Z^{l-1}\ra X^r_N$ is the canonical inclusion. Denote 
$$v: Z^l\cap Z^{l-1}\ra {\pi_i}^{-1}(\bZ),\quad u: {\pi_i}^{-1}(\bZ)\ra X^r_N$$
the canonical inclusions. Then $i_{Z^l\cap Z^{l-1}}=u\circ v$. Denote $\rho={\pi_i}_{|{\pi_i}^{-1}(\bZ)}$. 
There is a commutative diagram
\begin{equation*}
\begin{CD}
{\pi_i}^{-1}(\bZ)=Z^1\cup\ldots\cup Z^l @>u>>  X_N^r  \\
@V{\rho}VV        @VV{\pi_i}V\\
\bZ  @>{i_{\bZ}}>>   X^r_{N\setminus i}
\end{CD}
\end{equation*}
The restriction maps ${\rho}_{|Z^t}: Z^t\ra\bZ$ are induced by the forgetful maps $\LM_{N_t\cup\{i\}}\ra \LM_{N_t}$ if $t<l$ and $X^r_{N_l\cup\{i\}}\ra X^r_{N_l}$ for $t=l$. 
Note that the restriction map ${\rho}_{|Z^l\cap Z^{l-1}}: Z^l\cap Z^{l-1}\ra\bZ$ is an isomorphism. Denote
$$\overline{\cM}=R\rho_*(Rv_*\cM)=G_{a_1}^\vee\boxtimes\ldots\boxtimes G_{a_{l-1}}^\vee\boxtimes {G^r_k}^\vee.$$
For all $1\leq t\leq l-1$ we have
\begin{equation}\label{pullback}
(\rho^*\overline{\cM})_{|Z^t}=G_{a_1}^\vee\boxtimes\ldots\boxtimes \pi_i^*({G^r_{a_t}}^\vee)\ldots\boxtimes G_{a_{l-1}}^\vee\boxtimes {G^r_k}^\vee.
\end{equation}
while 
\begin{equation}\label{pullback2}
(\rho^*\overline{\cM})_{|Z^l}=G_{a_1}^\vee\boxtimes\ldots\boxtimes {G^r_{a_t}}^\vee\ldots\boxtimes G_{a_{l-1}}^\vee\boxtimes \pi_i^*({G^r_k}^\vee).
\end{equation}

The strata $Z^1,\ldots, Z^l$ intersect: if $t<s$, then $Z_t\cap Z_s\neq\emptyset$ if and only if $s=t+1$. 
There are exact sequences:
$$0\ra \cO_{Z^1\cup\ldots\cup Z^{l-1}}(-Z^l)\ra\cO_{Z^1\cup\ldots\cup Z^l}\ra\cO_{Z^l}\ra 0,$$
$$0\ra \cO_{Z^1\cup\ldots\cup Z^{l-2}}(-Z^{l-1})\ra\cO_{Z^1\cup\ldots\cup Z^{l-1}}(-Z^l)\ra\cO_{Z^{l-1}}(-Z^l)\ra 0,$$
$$\vdots$$
$$0\ra \cO_{Z^1}(-Z^2)\ra\cO_{Z^1\cup Z^2}(-Z^3)\ra\cO_{Z^2}(-Z^3)\ra 0.$$

We also consider the exact sequence:
$$0\ra \cO_{Z^l}(-Z^{l-1})\ra\cO_{Z^l}\ra\cO_{Z^l\cap Z^{l-1}}\ra 0.$$

We tensor all the above sequences with $\rho^*\overline{\cM}$. If we denote
$$\cN^t=\rho^*\overline{\cM}\otimes\cO_{Z^t}(-Z^{t+1}) \quad (1\leq t\leq l-1),\quad \cN^0=\rho^*\overline{\cM}\otimes\cO_{Z^l}(-Z^{l-1}),$$
$$\cF^t=\rho^*\overline{\cM}\otimes\cO_{Z^1\cup\ldots\cup Z^t}(-Z^{t+1})\quad (1\leq t\leq l-1),$$
we have exact sequences on $Z^1\cup\ldots\cup Z^l$:
$$0\ra\cF^{l-2}\ra \rho^*\overline{\cM}\ra (\rho^*\overline{\cM})_{|Z^l}\ra0,$$
$$0\ra\cF^{l-3}\ra\cF^{l-2}\ra\cN^{l-1}\ra0,$$
$$0\ra\cF^{l-4}\ra\cF^{l-3}\ra\cN^{l-2}\ra0,$$
$$\vdots$$
$$0\ra\cF^1=\cN^1\ra\cF^2\ra\cN^2\ra0,$$
and furthermore,
$$0\ra\cN^0\ra(\rho^*\overline{\cM})_{|Z^l}\ra v_*\cM\ra0,$$

Consider the push-forwards via $u_*(-)$ to  to $X^r_N$  of all of the above exact sequences. 
Recall that $\tilde{\cT}=u_*(v_*\cM)$. 
To prove that  $\Cone\left [{L\pi_i}^*{R\pi_i}_*\tilde{\cT}\to \tilde{\cT}\right]$ is in the subcategory generated by 
$\hat \bG_N^{r}$, it suffices to prove that for $\cN$ among 
$$\rho^*\overline{\cM},\quad \cN^1,\quad\ldots\quad,\quad\cN^{l-1},\quad \cN^0,$$ 
we have that $\Cone\left [{L\pi_i}^*{R\pi_i}_*(u_*\cN)\to (u_*\cN)\right]$ is in the subcategory generated by 
$\hat \bG_N^{r}$.  This is clear for $\rho^*\overline{\cM}$, as $u_*\rho^*\overline{\cM}=\pi_i^*{i_{\bZ}}_*\overline{\cM}$ 
(since $\pi_i$ is flat) and we have $\Cone\left [{L\pi_i}^*{R\pi_i}_*L\pi_i^*A\to L\pi_i^*A\right]=0$ for any $A$. 
As 
$$\cO_{Z^t}(-Z^{t+1})=\cO_{\LM_{N_1}}\boxtimes\ldots\boxtimes\cO_{\LM_{N_t}}(-\de_{i,y})\boxtimes\ldots\boxtimes\cO_{X^r_{N_l\cup\{i\}}}$$ 
where $y$ is one of the attaching points of $\LM_{N_t}$, using (\ref{pullback}) and Lemma \ref{cuspidal}, it follows that 
$u_*\cN^t$ is in the subcategory generated by $\hat \bG_N^{r}$. In particular,  ${R\pi_i}_*(u_*\cN^t)=0$ and 
$\Cone\left [{L\pi_i}^*{R\pi_i}_*(u_*\cN^t)\to (u_*\cN^t)\right]=u_*\cN^t$. 
Similarly, $u_*\cN^0$ is in the subcategory generated by $\hat \bG_N^{r}$ since 
$$\cO_{Z^l}(-Z^{l-1})=\cO_{\LM_{N_1}}\boxtimes\ldots\boxtimes\cO_{X^r_{N_l\cup\{i\}}}(-E_i)$$ 
and we may use (\ref{pullback2}) and Lemma \ref{cuspidal}. 
\ep


\end{document}